\documentclass[12pt,reqno]{amsart}
\usepackage{graphicx}
\usepackage{amssymb,amsmath,amsthm}
\usepackage{amsthm}
\usepackage{color}
\usepackage[pdf]{pstricks}
\usepackage{hyperref}
\usepackage{empheq}
\usepackage{pgfplots}
\usepackage[utf8]{inputenc}

\numberwithin{equation}{section}

\newtheorem{remark}{Remark}

\newtheorem{proposition}{Proposition}
\newtheorem{lemma}{Lemma}
\newtheorem{theorem}{Theorem}

\title[Bifurcations of coupled solitary waves]{Bifurcations of solitary waves \\ 
in a coupled system of long and short waves}

\author[J. Hornick]{James Hornick}
\address[J. Hornick]{Department of Mathematics and Statistics, McMaster University, Hamilton, Ontario, Canada, L8S 4K1}
\email{hornickj@mcmaster.ca}

\author[D. E. Pelinovsky]{Dmitry E. Pelinovsky}
\address[D. E. Pelinovsky]{Department of Mathematics and Statistics, McMaster University, Hamilton, Ontario, Canada, L8S 4K1}
\email{pelinod@mcmaster.ca}

\begin{document}

\maketitle

\begin{abstract}
	We consider families of solitary waves in the Korteweg--de Vries (KdV) equation coupled with the linear Schr\"{o}dinger (LS) equation. This model has been used to describe interactions between long and short waves. To characterize families of solitary waves, we consider a sequence of local (pitchfork) bifurcations of the uncoupled KdV solitons. The first member of the sequence is the KdV soliton coupled with the ground state of the LS equation, which is proven to be the constrained minimizer of energy for fixed mass and momentum. The other members of the sequence are the KdV solitons coupled with the excited states of the LS equation. We connect the first two bifurcations with the exact solutions of the KdV--LS system frequently used in the literature. 
\end{abstract}

\section{Introduction} 

Coupling between long nonlinear dispersive waves and short linear high-frequency wave packets can be modeled by using the coupled system of the Korteweg-de Vries (KdV) equation and the linear Schr\"{o}dinger (LS) equation:
\begin{equation}
\label{KdV-NLS-start}
\begin{cases}
u_t+\alpha uu_x+\beta u_{xxx}+\gamma (\vert \psi\vert^2)_x=0, \\
i\psi_t+\kappa \psi_{xx}+\sigma u\psi = 0,
\end{cases}
\end{equation}
where \(\alpha, \beta, \gamma, \kappa, \sigma\) are nonzero real constants, which depend on the small parameter of the original physical problem. 

The coupled KdV--LS system (\ref{KdV-NLS-start}) was proposed to describe various physical models including the resonant interaction of capillary-gravity waves \cite{kaw}, interaction between short surface waves and long internal waves \cite{Hashizume1980}, the electron propagation coupled to nonlinear ion-acoustic waves in a collisionless plasma \cite{Nish}, and more recently, the energy transport by electrons in anharmonic lattices \cite{Cis}. The KdV--LS system (\ref{KdV-NLS-start}) was further used for modeling of the interaction between internal and surface waves in a two-layer stratified fluid 
\cite{Sul,Sul15} (see also \cite{Choi}). Long internal solitons are modeled by the KdV equation and the short modulated surface waves are modeled by the LS equation. The interaction provides a tool to detect internal waves from the fluid surface \cite{Sul}. Further extensions of the coupled models between long and short waves  were constructed recently for the deep water in \cite{Sul24}.

Derivation of the KdV--LS system (\ref{KdV-NLS-start}) was reviewed in \cite{Nguyen}, where it was pointed out that the presence of different scales among the coefficients of the model makes the rigorous derivation challenging. See also the follow-up work \cite{Nguyen2} for the study of well-posedness of the coupled models. By using the coordinate transformation 
$$
x\rightarrow \lambda_{1} x, \quad t \rightarrow \lambda_{2} t, \quad u\rightarrow \lambda_{3} u, \quad \psi\rightarrow \lambda_{4}\psi
$$
with 
$$
\lambda_1 = \frac{\kappa}{\beta}, \quad \lambda_2 = \frac{\kappa^3}{\beta^2}, \quad \lambda_3 = \frac{\alpha \beta}{\kappa^2}, \quad 
\lambda_4 = \frac{\sqrt{|\alpha \gamma |} \beta^2}{\kappa^2},
$$
the original coupled system (\ref{KdV-NLS-start}) can be transformed to the normalized form:
\begin{equation}
    \label{KdV-NLS}
\begin{cases}
    u_{t} + uu_{x} + u_{xxx} + s(\vert \psi\vert^2)_{x}=0, \\
 i\psi_{t}+\psi_{xx}+ku\psi=0,
\end{cases}
\end{equation}
where
\[
k = \frac{\sigma \beta}{\alpha \kappa} \in \mathbb{R}, \quad \mbox{\rm and} \quad  
s = \text{sgn}(\alpha \gamma) = \pm 1.
\]
If $\kappa = \mathcal{O}(\varepsilon^{-1})$, $\sigma = \mathcal{O}(\varepsilon^{-1})$ and $\gamma = \mathcal{O}(\varepsilon^{2p})$ in terms of the physically relevant small parameter $\varepsilon > 0$ with some $p \in (0,1)$, see \cite{Sul}, then $k = \mathcal{O}(1)$ is independent of $\varepsilon$, so that the normalized form (\ref{KdV-NLS}) is also independent of  $\varepsilon$. Similarly, the coupled model for the capillary--gravity waves with additional transport terms derived in \cite{kaw} is transformed to the 
normalized form (\ref{KdV-NLS}) by using the Galilean transformation. 

The coupled KdV--LS system (\ref{KdV-NLS}) is Hamiltonian with the following conserved quantities:
\begin{align}
\label{mass}
Q(\psi) &= \frac{s}{k} \int \vert \psi\vert^2dx, \\
\label{momentum}
P(u,\psi) &= \frac{1}{2} \int\left( u^2 + \frac{i s}{k} (\overline{\psi}\psi_x - \psi\overline{\psi}_x) \right)dx, \\ 
\label{energy}
H(u,\psi) &= \frac{1}{2} \int\left( (u_{x})^2 - \frac{1}{3} u^3 + \frac{2 s}{k}\vert \psi_{x}\vert^2 - 2 s u\vert \psi\vert^2\right)dx,
\end{align}
which have the physical meaning of mass, momentum, and energy, respectively.  

The coupled KdV--LS system (\ref{KdV-NLS}) admits a family of uncoupled KdV solitons, for which $\psi = 0$. Existence of coupled solitary waves with 
$\psi \neq 0$ was studied recently in \cite{Anco,Nguyen3}. Stability 
of the explicit family of coupled solitary waves was shown in \cite{Chen}. 
In the case $s = -1$ and $k < 0$, coupled solitary waves 
were studied in \cite{Pava,Albert13,Pava2006} by using the concentration compactness method, from which the existence and orbital stability of a global minimizer of energy \(H\) for fixed mass \(Q\) and momentum \(P\) follow. 
Within the variational methods, it is difficult to clarify the admissible 
values of the wave speed and frequency, for which the constrained minimizer of energy is realized, as well as the corresponding profile $(u,\psi)$ of the coupled solitary waves.

{\em The main purpose of our work is to clarify the existence and stability 
of families of coupled solitary waves by studying their local bifurcations from 
the family of uncoupled KdV solitons for $s = +1$ and $k > 0$.} This approach 
allows us to connect the two families of coupled solitary waves obtained 
in \cite{Anco} to the first two local (pitchfork) bifurcations in a sequence of 
bifurcations of the family of the uncoupled KdV solitons.  

By using the Lyapunov–Schmidt reduction, we not only prove the existence of families of coupled solitary waves within the admissible values of the wave speed and frequency, but also we study the Morse index of the Hessian operator associated with the variational formulation of the wave profile $(u,\psi)$ as a constrained minimizer of energy for fixed mass and momentum. As the main outcome of our analysis, we prove that the coupled {\em sign-definite} solitary wave is 
a constrained minimizer of energy and the coupled {\em sign-indefinite} solitary waves are saddle points of the constrained energy. As a part of explicit computations, we also recover the stability conclusion for the exact sign-definite solitary waves obtained in \cite{Chen}. 

It is interesting that our conclusions remain valid for both supercritical and subcritical pitchfork bifurcations of coupled solitary waves. Based on numerical approximations, we show that both the supercritical and subcritical pitchfork 
bifurcations do occur for the first two bifurcations, depending on the parameter 
$k > 0$ in the system (\ref{KdV-NLS}) with $s = +1$. These results agree with the qualitative 
studies of pitchfork bifurcations (among the other bifurcations) under the presence of symmetries in the generalized NLS equation \cite{Yang}.

The paper is organized as follows. Section \ref{sec-2} presents our 
main results on the existence, bifurcations, and variational characterization 
of the solitary waves in the coupled KDV--LS system (\ref{KdV-NLS}). Section \ref{sec-3} gives details of analysis of the 
uncoupled KdV solitons with precise computations of the Morse index 
and the sequence of bifurcation points. Sections \ref{sec-4} and \ref{sec-5} report on the analysis of the first two pitchfork bifurcations. The analysis in Sections \ref{sec-3}, \ref{sec-4}, and \ref{sec-5} provides the proof of the main results.

\section{Main results}
\label{sec-2}

\subsection{Existence of traveling waves}

We consider the traveling wave solutions of the coupled KdV--LS system (\ref{KdV-NLS}) in the form:
\begin{equation}
    \label{TW}
u(x,t) = U(\xi), \quad \psi(x,t) =  e^{-i\omega t} \Psi(\xi), \quad \xi = x-ct,
\end{equation}
where the profiles $U : \mathbb{R} \to \mathbb{R}$ and $\Psi : \mathbb{R} \to \mathbb{C}$ satisfy the 
following system of differential equations:
\begin{equation}
    \label{ode-system-1}
    \begin{cases}
U''' - cU' + U U' + s (\vert \Psi\vert^2)'=0, \\
\Psi'' - ic \Psi' + k U \Psi +\omega \Psi=0.
\end{cases}
\end{equation}
Integrating the first equation of the system (\ref{ode-system-1}), we obtain 
\begin{equation}
    \label{ode-system-2}
U'' - c U + \frac{1}{2} U^2 + s \vert \Psi\vert^2 = C_1,
\end{equation}
where $C_1$ is an integration constant. If $U(\xi) \to 0$ and $|\Psi(\xi)|^2 \to 0$ as $|\xi| \to \infty$, then $C_1 = 0$. To integrate the second equation of the system (\ref{ode-system-1}), we use the 
polar form $\Psi = A e^{i\Theta}$ and obtain 
\begin{equation}
    \label{ode-system-3}
\begin{cases}
A'' + (\omega + kU) A + \Theta' (c - \Theta') A = 0, \\ 
\Theta'' A + 2 \Theta' A' - c A' =0. 
\end{cases}
\end{equation}
Multiplying the second equation of the system (\ref{ode-system-3}) by $A$ and integrating, we obtain 
\[
\Theta' = \frac{C_2}{A^2} + \frac{c}{2},
\]
where $C_2$ is another integration constant. If $A(\xi) \to 0$ as $|\xi| \to \infty$, then we have to choose 
$C_2 = 0$ to avoid divergence of $\Theta'(\xi)$ at infinity, which yields $\Theta'(\xi) =\frac{c}{2}$. 
Equation (\ref{ode-system-2}) and the first equation of the system (\ref{ode-system-3}) are now written 
as the system of two second-order differential equations for the profiles $U : \mathbb{R} \to \mathbb{R}$ and $A : \mathbb{R} \to \mathbb{R}$:
\begin{equation}
    \label{ode-system}
\begin{cases} 
U'' - c U + \frac{1}{2} U^2 + s A^2=0, \\ 
A'' + \left( \Omega + k U \right) A=0,
\end{cases}
\end{equation}
where $\Omega := \omega  +\frac{c^2}{4}$. The system (\ref{ode-system}) has the first invariant:
\begin{equation}
\label{first}
H(U,A) =  \frac{1}{2} (U')^2 - \frac{c}{2} U^2 + \frac{1}{6} U^3 + \frac{s}{k} (A')^2 + s U A^2 + \frac{\Omega s}{k}  A^2 = E,
\end{equation}
where $E$ is constant along solutions of the system (\ref{ode-system}). 
The first invariant (\ref{first}) gives the energy of the degree-two Hamiltonian system expressed by (\ref{ode-system}). 
If the system admits the second invariant, then it is Liouville integrable \cite{Hietarinta}. 

Analysis of the exact solitary wave solutions to the system (\ref{ode-system}) was developed in \cite{Anco} (see also \cite{Nguyen3}). If $A = 0$, the system reduces to the scalar equation 
\begin{equation}
\label{ode-kdv}
U'' - c U + \frac{1}{2} U^2 = 0,
\end{equation}
which is the traveling wave reduction of the integrable KdV equation.
Solving (\ref{ode-kdv}) yields the uncoupled KdV soliton for arbitrary $k$,
\begin{equation}
\label{solution-1}
U = 3 c \; \text{sech}^2\left(\frac{\sqrt{c}}{2}\xi\right), \quad A=0,
\end{equation}
where $c > 0$ is assumed. 

If $k = \frac{1}{6}$, 
the system (\ref{ode-system}) is obtained from the traveling wave reduction of the integrable 
(Melnikov) system derived in \cite{Mel1} and analyzed in \cite{Mel3,Mel2}:
\begin{equation}
\label{pde-system}
\begin{cases} 
(u_t + u u_x + u_{xxx})_x - 3 u_{yy} + s (|\psi|^2)_{xx} =0, \\ 
-i \psi_y + \psi_{xx} + \frac{1}{6} u \psi = 0.
\end{cases}
\end{equation}
Then, $u(x,y,t) = U(x-ct)$ and $\psi(x,y,t) = e^{i \Omega y} A(x-ct)$ reduces (\ref{pde-system}) to (\ref{ode-system}) with 
$k = \frac{1}{6}$. Exact solutions for the coupled solitons of the Melnikov system (\ref{pde-system}) were  obtained in \cite{Mel2}. For $k = \frac{1}{6}$, the coupled soliton of the system (\ref{ode-system}) obtained from (\ref{pde-system}) is given by 
\begin{equation}
\label{solution-2}
\begin{cases}
U=-12 \Omega \; \text{sech}^2(\sqrt{-\Omega} \xi), \\
A=\sqrt{-12 s \Omega (c+4\Omega)} \; \text{sech}(\sqrt{-\Omega} \xi),
\end{cases}
\end{equation}
where $\Omega <0$ is assumed. The solution (\ref{solution-2}) is well-defined if $s (c + 4 \Omega) > 0$. If $c > 0$, this can be satisfied with two options:
\begin{enumerate}
	\item[(a)] $s = 1$: $\Omega > -\frac{c}{4}$,
	\item[(b)] $s = -1$: $\Omega < -\frac{c}{4}$.
\end{enumerate}
The exact solution (\ref{solution-2}) is widely used in the literature, e.g. \cite{Pava,Chen}. It exists because the system (\ref{ode-system}) is Liouville integrable for $k=\frac{1}{6}$. Indeed, by using (3.2.26) and (3.2.27) in \cite{Hietarinta}, we obtain the second invariant of the degree-two Hamiltonian system:
\begin{equation}
\label{second-invariant}
I(U,A) = A (A') (U') - U (A')^2 + A^2 \left(\Omega U + \frac{1}{12} U^2 + \frac{s}{4} A^2 \right) + 3 (c + 4 \Omega) ((A')^2 + \Omega A^2).  
\end{equation}

In addition to (\ref{solution-1}) and (\ref{solution-2}), there exists another family of exact solutions of the system (\ref{ode-system}) based on \cite{Anco}. 
This additional family is not related to the two integrable cases. It exists for varying values of the parameter $k$ of the system (\ref{ode-system}):
	\begin{equation}
	\label{solution-3}
	k = -\frac{3 \Omega}{c-2\Omega}, \qquad 
	\begin{cases}
	U = 2 (c-2 \Omega) \; \text{sech}^2(\sqrt{-\Omega} \xi), \\
	A=\sqrt{2 s (c+4\Omega)(c - 2 \Omega)} \; \text{sech}(\sqrt{-\Omega} \xi) \; \text{tanh}(\sqrt{-\Omega} \xi),
	\end{cases}
	\end{equation}
	where $c > 0$ and $\Omega < 0$ are again assumed so that $c - 2 \Omega > 0$. The solution (\ref{solution-3}) is well-defined if $s (c + 4 \Omega) > 0$, which is the same condition as for the solution (\ref{solution-2}) with the same two options (a) and (b). The family (\ref{solution-3}) for $\Omega = -c$ and $s = -1$ corresponds to $k = 1$, for which the second invariant of the  system (\ref{ode-system}) can be obtained by using (3.2.22) in \cite{Hietarinta}:
	\begin{equation}
	\label{second-invariant-another}
	I(U,A) =  (A') (U') - c A U + \frac{s}{3} A^3 + \frac{1}{2} A U^2.
	\end{equation}
We did not find the second invariant of the system (\ref{ode-system}) for 
other values of $\Omega$ and $s$ in the solution (\ref{solution-3}).

\begin{remark}
	The exact solutions (\ref{solution-2}) and (\ref{solution-3}) exist for $c < 0$ if $s = -1$ (and $c - 2 \Omega > 0$ for the solution (\ref{solution-3})). However, we do not consider the values of $c < 0$ since we would like to connect the exact solutions (\ref{solution-2}) and (\ref{solution-3}) to the uncoupled solitons in the form (\ref{solution-1}) by means of local bifurcations.
\end{remark}

\begin{remark}
	\label{rem-reversibility}
	The system (\ref{ode-system}) enjoys the following reversibility symmetry: If $U(\xi)$ and $A(\xi)$ is the solution profile, then so are $U(-\xi)$ and $\pm A(-\xi)$. The solution is invariant under the reversibility symmetry if $U$ is spatially even and 
	$A$ is either spatially even or odd. The exact solutions (\ref{solution-1}), (\ref{solution-2}), and (\ref{solution-3}) are all invariant under the reversibility symmetry. Due to the translational symmetry, the representations (\ref{solution-1}), (\ref{solution-2}), and (\ref{solution-3}) can be translated from the point of symmetry at $\xi = 0$ to any point $\xi = \xi_0 \in \mathbb{R}$. 
\end{remark}

\subsection{Variational characterization}

The profile $(U,\Psi)$ defined by the system (\ref{ode-system-1}) is a critical point of the augmented energy functional:
\begin{equation}
    \label{energy-aug}
\Lambda(U,\Psi) = H(U,\Psi) + c P(U,\Psi) - \omega Q(\Psi).
\end{equation}
Indeed, applying variational derivatives to (\ref{energy-aug}) yields the system 
\begin{equation*}
\begin{cases}
-U'' - \frac{1}{2} U^2 - s \vert \Psi\vert^2 + c U =0, \\
-\frac{s}{k} \Psi'' - s U \Psi + \frac{i c s}{k}\Psi' - \frac{\omega s}{k} \Psi=0,
\end{cases}
\end{equation*}
which recovers (\ref{ode-system-1}) due to (\ref{ode-system-2}) with $C_1 = 0$. 

By using the representation $\Psi(\xi) = A(\xi) e^{\frac{ic \xi}{2}}$, we obtain the profile $(U,A)$ from the system (\ref{ode-system}). The profile 
$(U,A)$ is a critical point of the action functional
\begin{align}
\Lambda(U,A e^{\frac{ic \xi}{2}}) &= H(U,A) + c P(U,0) - \Omega Q(A) \notag \\
&= \frac{1}{2} \int \left(  (U')^2 + c U^2 - \frac{1}{3} U^3 + \frac{2 s}{k} (A')^2 - 2 s U A^2 - \frac{2 \Omega s}{k}  A^2 \right)d\xi,
    \label{action}
\end{align}
with $\Omega = \omega  +\frac{c^2}{4}$. Indeed, the representation $\Psi(\xi) = A(\xi) e^{\frac{ic \xi}{2}}$ transforms $\Lambda(U,\Psi)$ in (\ref{energy-aug}) to the form (\ref{action}), variational derivatives of which generate the system (\ref{ode-system}).

The following theorem gives the main result of this work. 

\begin{theorem}
	\label{theorem-main}
	Assume $c > 0$, $\Omega < 0$, and $s = {\rm sgn}(k)$. 
	If $s = -1$ or if $s = 1$ and $\Omega < \Omega_c$ with 
	\begin{equation}
	\label{Omega-c-intro}
	\Omega_c = -\frac{c}{16} \left( \sqrt{1 + 48 k} - 1 \right)^2,
	\end{equation}
	then the uncoupled soliton (\ref{solution-1}) is a local minimizer of the constrained energy $H$ for fixed momentum $P$ degenerate only by the translational symmetry. If $s = 1$ and  $\Omega \in (\Omega_c,0)$, 
	then the uncoupled soliton (\ref{solution-1}) is a saddle point of the constrained energy. Furthermore, for $k > \frac{J(J-1)}{12}$ with $J \in \mathbb{N}$, there exist a sequence $\{ \Omega_c^{(j)} \}_{j=1}^J$ of pitchfork bifurcations 
	(either super-critical or sub-critical) with 
		\begin{equation}
	\label{Omega-bif-intro}
	\Omega_c^{(j)} = -\frac{c}{16} \left( \sqrt{1 + 48 k} - 2j + 1 \right)^2, 
	\quad 1 \leq j \leq J, 
	\end{equation} 
	such that new families bifurcate from the family of uncoupled solitons. The family with $j = 1$ is a local minimizer of the constrained energy $H$ for fixed momentum $P$ and mass $Q$ degenerate only by the translational and rotational symmetries, whereas the families with $j = 2,\dots,J$ are saddle points of the constrained energy.
\end{theorem}

\begin{remark}
	We use $s = {\rm sgn}(k)$ to ensure that the second variation of the action functional (\ref{action}) at the family of the solitary waves with the profile $(U,A)$ be bounded from below (but not from the above). 
	\label{rem-values-of-s}
\end{remark}

\begin{remark}
	In the case of the local minimizers of the constrained energy, the orbital stability of either the uncoupled soliton (\ref{solution-1})  
	or the coupled solitary wave for perturbations $(w,z)$ in $H^1(\mathbb{R},\mathbb{R}) \times H^1(\mathbb{R},\mathbb{C})$ 
	follows from a general argument in Theorem 2.8 in \cite{Dmitry}. However, for the case of the saddle points of the constrained energy, 
	the orbital instability does not follow immediately unless the spectral instability can be proven, see Theorem 2.4 in \cite{Dmitry}. We do not have the spectral instability of the uncoupled soliton (\ref{solution-1}) 
	since the KdV equation is quadratically coupled to the solution of the 
	linear Schr\"{o}dinger equation in the KdV--LS system (\ref{KdV-NLS}). 
	Hence, we are not able to conclude on the 
	orbital instability of the uncoupled soliton (\ref{solution-1}) if $s = 1$ and $\Omega \in (\Omega_c,0)$.
\end{remark}

\subsection{Schematic illustration and the road map}

Figure \ref{fig-bifurcations} presents the bifurcation diagram for the three families of solitary wave solutions of the system (\ref{ode-system}) with 
$s = 1$, $k = \frac{1}{2}$, and a fixed value of $c > 0$. The bifurcation parameter is $\Omega < 0$. The family of uncoupled KdV solitons in the form (\ref{solution-1}) corresponds to the horizontal line and defines {\em the primary branch} for pitchfork bifurcations. In Proposition \ref{theorem-uncoupled}, we show that the Morse index $n(\widehat{\mathcal{L}})$ 
(the number of negative eigenvalues of the Hessian operator constrained by two symmetries of the KdV--LS system) for the primary branch is $0$ for $\Omega < -c$, 
$2$ for $-c < \Omega < -\frac{c}{4}$, and $4$ for $-\frac{c}{4} < \Omega < 0$ (if $k = \frac{1}{2}$). For $\Omega < -c$, the family of uncoupled KdV 
solitons is a local minimizer of the constrained energy $H$ for fixed momentum $P$, hence it is orbitally stable 
in the time evolution of the KdV--LS system (\ref{KdV-NLS}). 

Bifurcations of new {\em secondary} families of coupled solitary waves 
occur at $\Omega = -c$ 
and $\Omega = -\frac{c}{4}$ when the nullity index (the multiplicity of the zero eigenvalue of the Hessian operator) exceeds 
the number of symmetries. In Proposition \ref{theorem-bifurcation-first} and Figure \ref{fig-projection}, we show that the first bifurcation is a supercritical 
pitchfork bifurcation (if $k = \frac{1}{2}$). Furthermore, 
in Propositions \ref{theorem-Morse-first} and \ref{theorem-stability-first}, 
we show that the Morse index of the bifurcating family is $0$ so that it is a local minimizer of the constrained energy $H$ 
for fixed momentum $P$ and mass $Q$. The exact solution (\ref{solution-2}) gives a global continuation 
of this bifurcation (but in the case $k = \frac{1}{6}$). 

Finally, in Proposition \ref{theorem-bifurcation-second} and Figure \ref{fig-projection_2}, we show that the second bifurcation is a subcritical 
pitchfork bifurcation (if $k = \frac{1}{2}$). Furthermore, 
in Propositions \ref{theorem-Morse-second} and \ref{theorem-stability-second}, 
we show that the Morse index of the bifurcating family is $2$ so that it is a saddle point of the constrained energy $H$ 
for fixed momentum $P$ and mass $Q$. The exact solution (\ref{solution-3}) bifurcates from the primary branch for this exact value of $k = \frac{1}{2}$. However, it is globally extended 
in the form (\ref{solution-3}) for different values of $k$ in $\left(0,\frac{1}{2}\right)$ that depend on the 
bifurcation parameter $\Omega$, see Figure \ref{fig-region-third}.

\begin{figure}[htp!]
	\centering
	\includegraphics[width=0.75\textwidth]{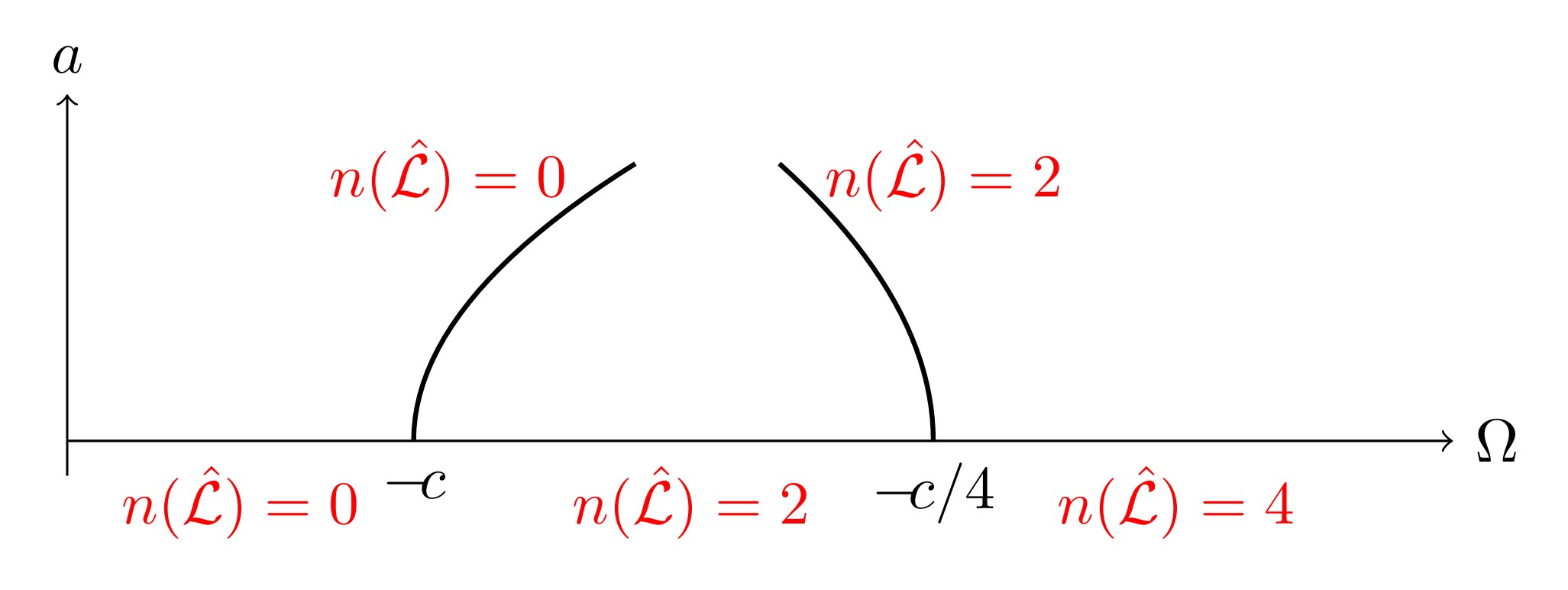}
	\caption{Schematic bifurcation diagram for $s = 1$ and $k = \frac{1}{2}$, which shows how the families of coupled solitary waves generalizing (\ref{solution-2}) and (\ref{solution-3}) are connected 
		with the family (\ref{solution-1}) of the uncoupled KdV solitons. $\widehat{\mathcal{L}}$ denotes the Hessian operator constrained by two symmetries of the KdV--LS system (\ref{KdV-NLS}) and $n(\widehat{\mathcal{L}})$ is its Morse index.}
	\label{fig-bifurcations}
\end{figure}

\begin{remark}
According to the reversibility symmetry in Remark \ref{rem-reversibility}, we obtain secondary branches in subspaces of functions with even and odd parities in the Sobolev space $H^2(\mathbb{R})$. We denote these subspaces by 
$H^2_{\rm even}(\mathbb{R})$ and $H^2_{\rm odd}(\mathbb{R})$, respectively. 
\end{remark}

\section{Stability and bifurcations of the uncoupled KdV solitons}
\label{sec-3}

\subsection{Hessian operator for the solitary waves}

The variational formulation of the solitary waves with the profile $(U,A)$ by using the action functional (\ref{action}) relates the system of differential equations (\ref{ode-system}) with the three conserved quantities of the KdV--LS system (\ref{KdV-NLS}) given by (\ref{mass}), (\ref{momentum}), and (\ref{energy}). We use this construction to define the Hessian 
operator for the solitary waves. The Hessian operator plays the central role in the analysis of stability and bifurcations of the solitary waves. 

\begin{lemma}
Let $(U,A) \in H^1(\mathbb{R}) \times H^1(\mathbb{R})$ be a critical point 
of the action functional (\ref{action}). The corresponding Hessian is expressed by the linear operator $\mathcal{L} : (H^2(\mathbb{R}))^3 \subset (L^2(\mathbb{R}))^3 \to (L^2(\mathbb{R}))^3$ given by 
\begin{equation}
    \label{Hessian}
\mathcal{L}=-\begin{pmatrix} \partial_{\xi}^2 + U - c & 2 s A & 0 \\ 
2 s A & \frac{2 s}{k}\left(\partial_{\xi}^2+kU+\Omega\right) & 0 \\
0&  0  &  \frac{2 s}{k}\left(\partial_{\xi}^2+kU+\Omega\right)
\end{pmatrix}.
\end{equation}
\label{lem-Hessian}
\end{lemma}

\begin{proof}
Adding a perturbation $(w,z)$ to the profile $(U,\Psi)$ in the augmented energy 
functional (\ref{energy-aug}) and using the traveling wave coordinate $\xi = x - ct$, we obtain the following expansion:
\begin{align*}
\Lambda(U+w,\Psi+z) &= \int\left( \frac{1}{2}(U'+w')^2-\frac{1}{6}(U+w)^3 
+ \frac{s}{k} \vert \Psi' + z' \vert^2 - s(U+w) \vert \Psi + z \vert^2 \right) d\xi \\
& \quad + c\int\left(\frac{1}{2}(U+w)^2 
+ \frac{i s}{2k} [(\overline{\Psi} + \bar{z})(\Psi' + z') - (\Psi + z)(\overline{\Psi}' + \overline{z}')] \right) d\xi \\
& \quad -\frac{\omega s}{k} \int \vert \Psi + z \vert^2 d \xi.
\end{align*}
Since $(U,\Psi)$ is a critical point of $\Lambda(U,\Psi)$, we get the expansion 
\begin{align*}
\Lambda(U+w,\Psi+z) = \Lambda(U,\Psi) + \frac{1}{2} Q_2(w,z) + \frac{1}{6} Q_3(w,z),
\end{align*}
where $Q_2$ and $Q_3$ are quadratic and cubic terms in $(w,z)$. To compute the Hessian operator, we only collect the quadratic terms in 
\begin{align*}
Q_2(w,z) &= \int\left( (w')^2 -  U w^2 
+ \frac{2 s}{k} \vert z' \vert^2 - 2 s U \vert z \vert^2 - 2 s w (\Psi \overline{z} + \overline{\Psi} z) \right) d \xi \\
& \quad + c\int\left( w^2 + \frac{i s}{k} (\bar{z} z' - z \overline{z}') \right) d\xi 
-\frac{2 \omega s}{k} \int \vert z \vert^2 d \xi.
\end{align*}
By using the variables 
\begin{equation}
    \label{variables-Psi-z}
\Psi = A e^{\frac{ic \xi}{2}}, \quad z = (z_1 + i z_2) e^{\frac{ic \xi}{2}},
\end{equation}
with real $A$ and $(z_1,z_2)$, we rewrite $Q_2(w,z)$ in the form
\begin{align*}
Q_2(w,z) &= \int\left( (w')^2 -  U w^2 
+ \frac{2 s}{k} [(z_1')^2 + (z_2')^2] - 2 s U (z_1^2 + z_2^2) - 4 s w A z_1 \right) d \xi \\
& \quad + c \int w^2 d\xi - \frac{2 \Omega s}{k} \int (z_1^2 + z_2^2) d \xi, 
\end{align*}
where $\Omega = \omega  +\frac{c^2}{4}$. Representing $Q_2(w,z)$ as a quadratic form for $\mathcal{L}$ acting on $(w,z_1,z_2)$ in $(L^2(\mathbb{R}))^3$
yields the Hessian operator $\mathcal{L}$ in the form (\ref{Hessian}). 
\end{proof}

It follows from (\ref{Hessian}) that $\mathcal{L}$ is block-diagonalized into a $2\times 2$ matrix Schr\"{o}dinger operator 
$\mathcal{L}_J : (H^2(\mathbb{R}))^2 \subset (L^2(\mathbb{R}))^2 \to (L^2(\mathbb{R}))^2$ given by 
\begin{equation}
    \label{block-1}
    \mathcal{L}_J = 
\begin{pmatrix} L_1 & -2 s A  \\ 
-2 s A & L_2 
\end{pmatrix} \qquad
\mbox{\rm with} \;\; 
\left\{ \begin{array}{l} L_1 = -\partial_{\xi}^2 + c - U, \\
L_2 = \frac{2 s}{k} \left( - \partial_{\xi}^2 - \Omega - kU \right),
\end{array} \right.
\end{equation}
and a scalar Schr\"{o}dinger operator $L_2 : H^2(\mathbb{R}) \subset L^2(\mathbb{R}) \to L^2(\mathbb{R})$. Since $U$ and $A$ decays to zero at infinity 
exponentially fast and we have assumed that $c > 0$ and $\Omega < 0$, Weyl's theorem implies that 
the continuous spectrum of $\mathcal{L}$ is a union of the continuous spectra of $L_1$ and $L_2$, which 
are given by $[c,\infty)$ and $\frac{2 s}{k} [|\Omega|,\infty)$, respectively. 
The continuous spectrum of $\mathcal{L}$ is strictly positive if and only if 
$$
s = {\rm sgn}(k),
$$
which is assumed from now on, see Theorem \ref{theorem-main} and Remark \ref{rem-values-of-s}. 

The main task of the stability analysis by using 
the Lyapunov theory is to compute the Morse index (the number and multiplicity of negative 
eigenvalues) and the degeneracy index (multiplicity of the zero eigenvalue) 
of the Hessian operator $\mathcal{L}$. To achieve the task, we recall 
the following well-known result, see \cite{Drazin}.

\begin{lemma}
Let $T : H^2(\mathbb{R}) \subset L^2(\mathbb{T}) \to L^2(\mathbb{R})$ be the scalar Schr\"{o}dinger operator given by 
\[
T=-\partial_x^2 - \gamma \;\mbox{\rm sech}^2(x), \quad \gamma>0.
\]
The continuous spectrum of \(T\) is \(\sigma_c(T)=[0,\infty)\), whereas the point spectrum of \(T\) depends on \(\gamma\) 
and consists of simple eigenvalues:
\begin{equation}
    \label{eigenvalues-Schr}
\sigma_p(T)=\left\{-\left( \frac{1}{2} \sqrt{1+4 \gamma} - n - \frac{1}{2}\right)^2, \ \ n=0,1,2,\dots,N\right \} \backslash \{0\},
\end{equation}
where $N = \lfloor \frac{1}{2} (\sqrt{1+4 \gamma} - 1) \rfloor$ with $\lfloor a\rfloor $ denoting the floor of a positive number $a$. 
\label{lem-Schr}
\end{lemma}

\begin{remark}
    If $\lfloor \frac{1}{2} (\sqrt{1+4 \gamma} - 1) \rfloor = \frac{1}{2} (\sqrt{1+4 \gamma} - 1)$, then $0$ is the resonance 
    of the Schr\"{o}dinger operator $T$ at the end point of $\sigma_c(T)=[0,\infty)$. It is excluded from $\sigma_p(T)$ in (\ref{eigenvalues-Schr}), 
    which then consists of $N$ negative eigenvalues. 
    If $\lfloor \frac{1}{2} (\sqrt{1+4 \gamma} - 1) \rfloor < \frac{1}{2} (\sqrt{1+4 \gamma} - 1)$, then $\sigma_p(T)$ in (\ref{eigenvalues-Schr}) 
    consists of $(N+1)$ negative eigenvalues.
\end{remark}

\subsection{Characterization of the uncoupled KdV solitons}

Lemmas \ref{lem-Hessian} and \ref{lem-Schr} give the necessary ingredients to study the variational characterization of the uncoupled KdV soliton with the profile $(U,0)$, where $U$ is given by (\ref{solution-1}). Since $A = 0$ for the uncoupled KdV soliton (\ref{solution-1}), the mass conservation $Q$ plays no role in the variational characterization.

\begin{proposition}
    \label{theorem-uncoupled}
Assume $c > 0$, $\Omega < 0$, and $s = {\rm sgn}(k)$. 
If $s = -1$ or if $s = 1$ and $\Omega < \Omega_c$ with 
\begin{equation}
    \label{Omega-c}
    \Omega_c = -\frac{c}{16} \left( \sqrt{1 + 48 k} - 1 \right)^2,
\end{equation}
then the uncoupled KdV soliton with the profile (\ref{solution-1}) is a local minimizer of the constrained energy $H$ 
for fixed momentum $P$ degenerate only by the translational symmetry. 
If $s = 1$ and $\Omega \in (\Omega_c,0)$, 
then the uncoupled KdV soliton is a saddle point of the constrained energy. 
\end{proposition}

\begin{proof} 
For the uncoupled KdV soliton with the profile (\ref{solution-1}), the Hessian operator (\ref{Hessian}) is diagonal with 
$(L_1,L_2,L_2)$, where $L_1$ and $L_2$ are defined in (\ref{block-1}). Putting \(y = \frac{\sqrt{c}}{2} \xi\), 
we convert $L_1$ and $L_2$ to the form used in Lemma \ref{lem-Schr}:
\begin{align}
\label{block-2}
\left\{ \begin{array}{l}
    L_1 = \frac{c}{4} \left( - \partial_y^2 + 4 - 12 {\rm sech}^2(y) \right), \\
    L_2 = \frac{c}{2 |k|} \left( - \partial_y^2 + \frac{4 |\Omega|}{c} - 12 k {\rm sech}^2(y) \right).
    \end{array} \right.
\end{align}
By Lemma \ref{lem-Schr}, the normalized operator $T_1 = - \partial_y^2 + 4 - 12 {\rm sech}^2(y)$
has three eigenvalues \(-5,0,3\) isolated from the continuous spectrum on $[4,\infty)$. 
Therefore, $L_1$ has a simple negative eigenvalue and a simple 
zero eigenvalue with the eigenfunction spanned by $U'$ due to the translational symmetry.

If $s = {\rm sgn}(k) = -1$, the normalized operator $T_2 = - \partial_y^2 + \frac{4 |\Omega|}{c} + 12 |k| {\rm sech}^2(y)$
has no isolated eigenvalues from the continuous spectrum on $\left[ \frac{4 |\Omega|}{c},\infty \right)$. 
If $s = {\rm sgn}(k) = +1$, the normalized operator $T_2 = - \partial_y^2 + \frac{4 |\Omega|}{c} - 12 k {\rm sech}^2(y)$
has the smallest eigenvalue at 
$$
\frac{4 |\Omega|}{c} - \frac{1}{4} \left( \sqrt{1 + 48 k} - 1 \right)^2. 
$$
For $\Omega < \Omega_c$, where $\Omega_c$ is given by (\ref{Omega-c}), it is strictly positive, 
whereas for $\Omega \in (\Omega_c,0)$, it is strictly negative. 

We conclude that for $s = -1$ or if $s = 1$ and $\Omega < \Omega_c$, 
the Morse index of the uncoupled KdV soliton with the profile $(U,0)$ in (\ref{solution-1}) is exactly $1$ and the degeneracy index 
is exactly $1$. By Theorem 2.7 in \cite{Dmitry}, $(U,0)$ is a local constrained minimizer of energy $H$ 
subject to fixed momentum $P$ degenerate only by the translational symmetry 
if and only if the slope condition is satisfied:
$$
\frac{d}{dc} P(U,0) > 0.
$$
This is true since 
\begin{equation}
    \label{momentum-zero}
P(U,0) = \frac{9c^2}{2} \int_{\mathbb{R}} {\rm sech}^4\left(\frac{\sqrt{c} \xi}{2}\right) d\xi = 12 c^{\frac{3}{2}}
\end{equation}
so that $\frac{d}{dc} P(U,0) > 0$ and the statement is proven.
If $s = 1$ and $\Omega \in (\Omega_c,0)$, the Morse index of the uncoupled KdV soliton with the profile $(U,0)$ in (\ref{solution-1}) is at least $3$. By Theorem 2.7 in \cite{Dmitry}, $(U,0)$ is a saddle point of the constrained energy. 
\end{proof}

\begin{remark}
Operator $L_2$ in (\ref{block-2}) admits zero eigenvalues at $\{ \Omega_c^{(n)} \}_{j=1}^{J}$, where $\Omega_c^{(j)}$ is given by (\ref{Omega-bif-intro}) with 
$\Omega_c^{(1)} \equiv \Omega_c$ and $J$ defined by the largest integer such that $k > \frac{J(J-1)}{12}$. At $\Omega = \Omega^{(j)}$, the primary branch of the uncoupled KdV solitons undergoes a local bifurcation and a new secondary branch of the coupled solitary waves appears. In Section \ref{sec-4} and \ref{sec-5}, we study the first two bifurcations, for which the new secondary branches include the exact solutions (\ref{solution-2}) and (\ref{solution-3}), respectively.
\end{remark}

\section{Bifurcation of the coupled solitary wave at $\Omega_c^{(1)} \equiv \Omega_c$}
\label{sec-4}

\subsection{The bifurcating branch of coupled solitary waves}

By Proposition \ref{theorem-uncoupled}, the first bifurcation along the family of the uncoupled KdV solitons occurs at $\Omega = \Omega_c$, for which the first positive eigenvalue of the Schr\"{o}dinger operator $L_2: H^2(\mathbb{R}) \subset L^2(\mathbb{R}) \to L^2(\mathbb{R})$ for $\Omega < \Omega_c$ passes through zero and becomes negative for $\Omega > \Omega_c$. The first eigenvalue of $L_2$ corresponds to the following eigenfunction:
\begin{equation}
\label{g-eigenfunction}
    g(\xi) = \text{sech}^{p}\left( \frac{\sqrt{c}}{2} \xi\right), \quad p = \frac{\sqrt{1+48 k}-1}{2}.
\end{equation}
Note that $\Omega_c = -\frac{c}{4}$ and $p = 1$ if $k = \frac{1}{6}$, for which 
the profiles $U$ in (\ref{solution-1}) and $g$ in (\ref{g-eigenfunction}) coincide
with the profile $(U,A)$ of the coupled solitary wave given by (\ref{solution-2}) for $\Omega = \Omega_c$. Since $s = {\rm sgn}(k) = +1$, the solution family (\ref{solution-2}) exists for $\Omega \in (\Omega_c,0)$, 
for which the solution family (\ref{solution-1}) is a saddle point of the constrained energy by Proposition \ref{theorem-uncoupled}.


The following result shows that the intersection of the branch of coupled solitary waves (\ref{solution-2}) with the branch of uncoupled KdV solitons (\ref{solution-1}) at $\Omega = \Omega_c$ is not a coincidence but the outcome of a local pitchfork bifurcation which is valid for every $k > 0$ near $\Omega = \Omega_c$. 

\begin{proposition}
    \label{theorem-bifurcation-first}
    Assume $c > 0$, $\Omega < 0$, and $s = {\rm sgn}(k) = 1$.
    Let $U_0$ be given by (\ref{solution-1}) and $g$ be given by (\ref{g-eigenfunction}). 
    Assume that $\langle g^2, L_1^{-1} g^2 \rangle \neq 0$, where $L_1$ is defined in (\ref{block-2}). 
    For every $\Omega$ near $\Omega_c$ such that ${\rm sgn}(\Omega - \Omega_c) = -{\rm sgn}(\langle g^2, L_1^{-1} g^2 \rangle)$, 
    there exists a unique family of solutions with the profile $(U,A) \in H^2_{\rm even}(\mathbb{R}) \times H^2_{\rm even}(\mathbb{R})$ satisfying 
    $$
    \| U - U_0 \|_{H^2} \leq C |\Omega - \Omega_c|, \quad \| A \| \leq C \sqrt{|\Omega - \Omega_c|},
    $$
    for some $\Omega$-independent constant $C > 0$.
\end{proposition}

\begin{proof}
We represent the system of equations (\ref{ode-system}) as the root finding problem for the vector field
\begin{equation}
    \label{vector-field}
     {\bf F}(U,A,\Omega) : H^2(\mathbb{R}) \times H^2(\mathbb{R}) \times \mathbb{R} \to L^2(\mathbb{R}) \times L^2(\mathbb{R}), 
\end{equation}
with 
\begin{equation*}
    {\bf F}(U,A,\Omega) = \begin{pmatrix} -U'' + c U - \frac{1}{2} U^2 - s A^2 \\  
    \frac{2s}{k} \left( -A'' - (\Omega + kU) A\right) \end{pmatrix}, 
\end{equation*}
where $c > 0$ and $k > 0$ are fixed and $s = {\rm sgn}(k) = 1$. The Jacobian 
of ${\bf F}$ at $(U,A)$ is given by the matrix Schr\"{o}dinger operator $\mathcal{L}_J : (H^2(\mathbb{R})^2 \subset (L^2(\mathbb{R}))^2 \to (L^2(\mathbb{R}))^2$ in (\ref{block-1}). 
Evaluating $\mathcal{L}_J$ at $(U,A,\Omega) = (U_0,0,\Omega_c)$ with $U_0$ given by (\ref{solution-1}) and 
using perturbation $(w,z,\delta \Omega)$ to $(U_0,0,\Omega_c)$ yields 
the expansion
\begin{equation}
    \label{vector-perturbed}
    {\bf F}(U_0+w,z,\Omega_c+\delta \Omega) = \underbrace{{\bf F}(U_0,0,\Omega_c)}_{=0} + 
    \begin{pmatrix} L_1 & 0  \\ 
0 & L_2 
\end{pmatrix} \begin{pmatrix} w \\ z \end{pmatrix} + \begin{pmatrix} - \frac{1}{2} w^2 - s z^2 \\  
    -2 s w z - \frac{2 s}{k} \delta \Omega z \end{pmatrix},
\end{equation}
where the explicit form of the Schr\"{o}dinger operators $L_1$ and $L_2$ is given in (\ref{block-2}) for $\Omega = \Omega_c$.
The root-finding problem for ${\bf F}$ can be rewritten in the form
\begin{equation}
    \label{fixed-point-1}
\begin{pmatrix} L_1 & 0  \\ 
0 & L_2 
\end{pmatrix} \begin{pmatrix} w \\ z \end{pmatrix} = \begin{pmatrix} \frac{1}{2} w^2 + s z^2 \\  
    2 s w z + \frac{2s}{k} \delta \Omega z \end{pmatrix}.
\end{equation}
We recall that $\text{ker}L_1 = {\rm span}(U_0')$ and $\text{ker} L_2 = {\rm span}(g)$, where $U_0'$ is odd and $g$ is even in $\xi$. 
To eliminate the translational symmetry, we consider the implicit equation (\ref{fixed-point-1}) on the subspace of even functions 
$(w,z) \in H^2_{\rm even}(\mathbb{R}) \times H^2_{\rm even}(\mathbb{R})$, for which $L_1$ is invertible with a bounded inverse. 
The implicit equation (\ref{fixed-point-1}) is closed on this subspace, thanks 
to the reversibility symmetry of the system (\ref{ode-system}), see Remark \ref{rem-reversibility}. Therefore, we use the orthogonal decomposition
\begin{equation}
    \label{LSdecomposition}
\begin{pmatrix} w \\ z \end{pmatrix} = a \begin{pmatrix} 0 \\ g\end{pmatrix} + 
\begin{pmatrix} w_1 \\ z_1 \end{pmatrix}, \quad \mbox{\rm such that } \; \langle g, z_1 \rangle = 0, 
\end{equation}
with $a \in \mathbb{R}$ being a (small) parameter defined by the orthogonality constraint with the standard inner product 
$\langle \cdot, \cdot \rangle$ in $L^2(\mathbb{R})$. 
The implicit equation (\ref{fixed-point-1}) is rewritten in the form:
\begin{equation}
    \label{fixed-point-2}
\begin{pmatrix} L_1 & 0  \\ 
0 & L_2 
\end{pmatrix} \begin{pmatrix} w_1 \\ z_1 \end{pmatrix} = \begin{pmatrix} \frac{1}{2} w_1^2 + s (ag+z_1)^2 \\  
    2 s w_1 (a g + z_1) + \frac{2 s}{k} \delta \Omega (a g + z_1) \end{pmatrix}.
\end{equation}
Orthogonality of the second equation of system (\ref{fixed-point-2}) to $\text{ker} L_2 = {\rm span}(g)$ yields the solvability condition 
\begin{equation}
    \label{solv-eq}
\delta \Omega a \| g \|^2_{L^2} + k  \langle g, w_1 (a g + z_1)\rangle = 0. 
\end{equation}
Under the constraint (\ref{solv-eq}), $L_2$ is invertible with a bounded inverse on a subspace of $L^2(\mathbb{R})$ orthogonal to 
$\text{ker} L_2 = {\rm span}(g)$. This allows us to rewrite the implicit equation (\ref{fixed-point-2}) in the final form:
\begin{equation}
    \label{fixed-point-3}
\begin{pmatrix} w_1 \\ z_1 \end{pmatrix} = \begin{pmatrix} L_1^{-1} & 0  \\ 
0 & L_2^{-1} 
\end{pmatrix} \begin{pmatrix} \frac{1}{2} w_1^2 + s (ag+z_1)^2 \\  
    2 s w_1 (a g + z_1) - 2 s  (a g + z_1) \frac{\langle g, w_1 (a g + z_1)\rangle}{a \| g \|^2_{L^2}}
    \end{pmatrix}.
\end{equation}
By the implicit function theorem, there exists a unique solution $(w_1,z_1) \in H^2_{\rm even}(\mathbb{R}) \times H^2_{\rm even}(\mathbb{R})$ 
to (\ref{fixed-point-3}) such that $\langle g, z_1 \rangle = 0$
for every (sufficiently small) $a \in \mathbb{R}$ and the mapping $a \mapsto (w_1,z_1)$ is $C^{\infty}$. Let us denote the unique solution as $(\mathfrak{w}_1(a),\mathfrak{z}_1(a))$. It follows from the principal terms of the system (\ref{fixed-point-3}) that 
the mapping satisfies 
\begin{equation}
    \label{bounds-on-terms}
\| \mathfrak{w}_1(a) \|_{H^2} \leq C a^2, \qquad \| \mathfrak{z}_1(a) \|_{H^2} \leq C |a|^3,
\end{equation}
for some $C > 0$ uniformly for small $a \in \mathbb{R}$. For precise computations, we need the following transformation 
\begin{equation}
    \label{near-identity}
\mathfrak{w}_1(a) = a^2 w_2 + \mathfrak{w}_2(a), \quad \| \mathfrak{w}_2(a) \|_{H^2} \leq C a^4, 
\end{equation}
where the $a$-independent function $w_2 \in H^2_{\rm even}(\mathbb{R})$ is given by $w_2 = s L_1^{-1} g^2$. Substituting 
$(\mathfrak{w}_1(a),\mathfrak{z}_1(a))$ satisfying (\ref{bounds-on-terms}) and (\ref{near-identity}) into (\ref{solv-eq}), 
we obtain 
\begin{equation}
    \label{delta-Omega}
\delta \Omega = - k a^2  \frac{\langle g^2, L_1^{-1} g^2 \rangle}{\| g \|^2_{L^2}} + \mathcal{O}(a^4),
\end{equation}
which shows that ${\rm sgn}(\delta \Omega) = -{\rm sgn}(\langle g^2, L_1^{-1} g^2 \rangle)$ since $k > 0$. This completes the proof.
\end{proof}

\subsection{Characterization of the bifurcating branch}

We first compute the Morse index and the nullity index of the bifurcating branch 
in Proposition \ref{theorem-bifurcation-first}. 

\begin{proposition}
    \label{theorem-Morse-first}
    Let $(U,A)$ be the profile of the bifurcating branch for $\Omega$ near $\Omega_c$ given by Proposition \ref{theorem-bifurcation-first}. 
Its Morse index is equal to $1$ if $\langle g^2, L_1^{-1} g^2 \rangle < 0$ and to $2$ if $\langle g^2, L_1^{-1} g^2 \rangle > 0$, 
    whereas the nullity index is equal to $2$ in both cases. 
\end{proposition}

\begin{proof}
We use the construction of Proposition \ref{theorem-bifurcation-first} with the profile 
\begin{equation}
\label{decomposition-bif}
U = U_0 + \mathfrak{w}_1(a) \quad  \mbox{\rm and} \quad 
A = a g + \mathfrak{z}_1(a),
\end{equation}
where $(\mathfrak{w}_1(a), \mathfrak{z}_1(a)) \in H^2_{\rm even}(\mathbb{R}) \times H^2_{\rm even}(\mathbb{R})$ satisfy the bounds (\ref{bounds-on-terms}). The parameter $a > 0$ parameterize the bifurcating branch, 
in particular, we have $\Omega = \Omega_c + \delta \Omega(a)$, where $\delta \Omega(a)$ is given by the expansion 
(\ref{delta-Omega}) with ${\rm sgn}(\delta \Omega) = -{\rm sgn}(\langle g^2, L_1^{-1} g^2 \rangle)$. 
The Hessian block $L_J$ defined by (\ref{block-1}) can be expanded by using (\ref{decomposition-bif}) as follows:
\begin{align*}
    L_J &= \begin{pmatrix} L_1 - \mathfrak{w}_1(a) & -2 s (a g + \mathfrak{z}_1(a))  \\ 
-2 s (a g + \mathfrak{z}_1(a)) & L_2 - 2 s \mathfrak{w}_1(a) -\frac{2 s}{k} \delta\Omega(a)
\end{pmatrix} \\
&= \begin{pmatrix} L_1 & 0  \\ 
0 & L_2 \end{pmatrix} 
+ a \begin{pmatrix} 0 & -2 s g \\ 
-2 s g & 0 \end{pmatrix} 
+ a^2 \begin{pmatrix} - s L_{1}^{-1} g^2 & 0 \\ 
0 & -2 L_{1}^{-1}g^2+  2\frac{\langle g^2, L_1^{-1} g^2 \rangle}{\| g \|^2_{L^2}}  
\end{pmatrix} + \mathcal{O}(a^3) \\
&= L_J^{(0)} + a L_J^{(1)} + a^2 L_J^{(2)} + \mathcal{O}(a^3),
\end{align*}
where we have used the leading-order terms for $\mathfrak{w}_1(a)$ and $\delta \Omega(a)$ from (\ref{near-identity}) and (\ref{delta-Omega}).
We are now looking at the eigenvalues and eigenvectors of the eigenvalue equation $L_J v = \lambda v$. 
Since $0$ is a double eigenvalue of $L_J^{(0)}$, we are looking for the splitting of the zero eigenvalue 
by using the perturbation theory. We write 
$$
v = v^{(0)} + a v^{(1)} + a^2 v^{(2)} + \mathcal{O}(a^3), \quad \lambda = a \lambda^{(1)} + a^2 \lambda^{(2)} + \mathcal{O}(a^3),
$$
where
$$
v^{(0)} = c_1 \begin{pmatrix} U_0'\\0\end{pmatrix} + c_2\begin{pmatrix}0\\g\end{pmatrix}
$$
for some $(c_1,c_2) \in \mathbb{R}^2$.

\underline{At the order of $\mathcal{O}(a)$,} we obtain 
\[
L^{(0)} v^{(1)} + L^{(1)} v^{(0)} =\lambda^{(1)} v^{(0)}.
\]
Writing in the component form, we obtain 
$$
v_1=\begin{pmatrix}f_1\\g_1\end{pmatrix} : \quad \left\{ \begin{array}{l} L_1 f_1 - 2 s g (c_2 g) = \lambda^{(1)} (c_1 U_0'), \\
L_2 g_1 - 2 s g (c_1 U_0') = \lambda^{(1)} (c_2 g). \end{array} \right.
$$
Since even $g^2$ is orthogonal to $U_0'$ and odd $g U_0'$ is orthogonal to even $g$, we obtain by Fredholm's theory 
for linear inhomogeneous equations that $\lambda^{(1)} = 0$. This yields the exact solution in the form 
$$
f_1 = 2 s c_2 L_1^{-1} g^2, \quad g_1 = 2 s c_1 L_2^{-1} g U_0',
$$
where $L_1^{-1}$ is uniquely defined in the space of even functions and $L_2^{-1}$ is uniquely defined in the space of odd functions. 
Differentiating of $(-\partial_{\xi}^2 - \Omega_c - k U_0) g = 0$ in $\xi$ yields $L_2 g' = 2 s U_0' g$, which ensures that 
$g_1 = c_1 g'$.

\underline{At the order of $\mathcal{O}(a^2)$,} we obtain 
\[
L^{(0)} v^{(2)} + L^{(1)} v^{(1)} + L^{(2)} v^{(0)} =\lambda^{(2)} v^{(0)}.
\]
Writing in the component form, we obtain 
$$
v_2=\begin{pmatrix}f_2\\g_2\end{pmatrix} : \quad \left\{ \begin{array}{l} L_1 f_2 - 2 s g g_1 - s  (L_1^{-1} g^2) (c_1 U_0') = \lambda^{(2)} (c_1 U_0'), \\
L_2 g_2 - 2 s g f_1 + \left( - 2 L_1^{-1} g^2 + 2 \frac{\langle g^2, L_1^{-1} g^2 \rangle}{\| g \|^2_{L^2}} \right) (c_2 g) = \lambda^{(2)} (c_2 g). \end{array} \right.
$$
Taking the inner product of the first equation with \(U_0'\in \text{Ker}(L_1)\) and 
also taking the inner product of the second equation with \(g\in \text{Ker}(L_2)\), we get the linear system 
for $(c_1,c_2) \in \mathbb{R}^2$:
\begin{align*}
-4 c_1 \langle g U_0', L_2^{-1} g U_0' \rangle - s c_1 \langle (U'_0)^2, L_{1}^{-1}g^2 \rangle &= \lambda^{(2)} c_1 \Vert U_{0}'\Vert^2_{L^2}, \\
-4c_2\langle g^2, L_{1}^{-1}g^2\rangle &= \lambda^{(2)} c_2 \Vert g\Vert^2_{L^2}
\end{align*}
where we have used the explicit expressions for $f_1$, $g_1$. Since the linear system is diagonal, we get two different solutions 
related to the subspaces $(c_1,c_2) = (1,0)$ and $(c_1,c_2) = (0,1)$. 

For the subspace $(c_1,c_2) = (1,0)$, we prove that 
\begin{equation}
    \label{constraint-transl}
4 \langle g U_0', L_2^{-1} g U_0' \rangle + s \langle (U'_0)^2, L_{1}^{-1}g^2 \rangle = 0,
\end{equation}
which yields $\lambda^{(2)} = 0$. This is consistent with the fact that the zero eigenvalue $\lambda = 0$ is 
preserved along the solution branch with the profile $(U,A)$ by the translational symmetry. To verify (\ref{constraint-transl}), we derive from 
\begin{align*}
    U_0'' &= c U_0 - \frac{1}{2} U_0^2, \\
    (U_0')^2 &= c U_0^2 - \frac{1}{3} U_0^3,
\end{align*}
that 
\begin{align*}
    L_1 U_0 &= -U_0'' + c U_0 - U_0^2 = -\frac{1}{2} U_0^2, \\
    L_1 U_0^2 &= -2 U_0 U_0'' - 2 (U_0')^2 + c U_0^2 - U_0^3 = -3 c U_0^2 + \frac{2}{3} U_0^3.
\end{align*}
Hence we obtain 
\begin{align*}
\langle (U'_0)^2, L_{1}^{-1}g^2 \rangle &= c \langle U_0^2, L_{1}^{-1}g^2 \rangle - \frac{1}{3} \langle U_0^3, L_{1}^{-1}g^2 \rangle\\
 &= c \langle L_{1}^{-1} U_0^2, g^2 \rangle - \frac{1}{3} \langle L_{1}^{-1} U_0^3, g^2 \rangle\\
  &= -\frac{c}{2} \langle L_{1}^{-1} U_0^2, g^2 \rangle - \frac{1}{2} \langle U_0^2, g^2 \rangle\\
    &= c \langle U_0, g^2 \rangle - \frac{1}{2} \langle U_0^2, g^2 \rangle.
\end{align*}
On the other hand, since $L_2^{-1} g U_0' = \frac{s}{2} g'$, we also obtain 
\begin{align*}
4 \langle g U_0', L_2^{-1} g U_0' \rangle &= 2 s \langle g U_0', g' \rangle \\
 &= - s \langle U_0'', g^2 \rangle \\
  &= -c s  \langle U_0, g^2 \rangle + \frac{s}{2} \langle U_0^2, g^2 \rangle.
\end{align*}
Substituting these computations into the left-hand side of (\ref{constraint-transl}), we obtain the zero result.

For the subspace $(c_1,c_2) = (0,1)$, we get 
\[
\lambda^{(2)}=-4\frac{\langle g^2,L_{1}^{-1}g^2\rangle}{\Vert g\Vert^2_{L^2}}
\]
which shows that $\lambda = a^2 \lambda^{(2)} + \mathcal{O}(a^3) > 0$ if $\langle g^2,L_{1}^{-1}g^2\rangle < 0$ and $\lambda = a^2 \lambda^{(2)} + \mathcal{O}(a^3) < 0$ if $\langle g^2,L_{1}^{-1}g^2\rangle > 0$. Since $L_J^{(0)}$ has only one negative eigenvalue at $U = U_0$, the splitting of the double zero eigenvalue 
of $L_J^{(0)}$ is clarified above, and the other eigenvalues of $L_J^{(0)}$ are strictly positive, the continuity of eigenvalues in $a$ implies that the Morse index of $L_J$ is $1$ if $\langle g^2,L_{1}^{-1}g^2\rangle < 0$ and $2$ if $\langle g^2,L_{1}^{-1}g^2\rangle > 0$ for small $a \neq 0$. 

Referring back to the Hessian operator $\mathcal{L}$ given by (\ref{Hessian}), it follows that 
the operator 
$$
L_2 = -\frac{2s}{k} (\partial_{\xi}^2 + k U + \Omega)
$$ 
admits a simple zero eigenvalue due to the rotational symmetry with $L_2 A = 0$. The rest of $L_2$ at $U = U_0$ and $\Omega = \Omega_c$ is strictly positive, 
hence the Morse index of $L_2$ is $0$ for small $a$. This yields the claim on the Morse index of $\mathcal{L}$.
On the other hand, $\mathcal{L}$ has a double zero eigenvalue associated with the translational and rotational symmetries, and the splitting 
of the double zero eigenvalue of $L_J^{(0)}$ clarified above shows that the kernel of $\mathcal{L}$ is exactly double for small $a$. This yields the 
claim about the nullity index of $\mathcal{L}$.
\end{proof}

We next clarify the constrained minimization properties of the bifurcating branch by using information about 
the Morse and nullity indices from Proposition \ref{theorem-Morse-first}.

\begin{proposition}
    \label{theorem-stability-first}
    Let $(U,A)$ be the profile of the bifurcating branch for $\Omega$ near $\Omega_c$ given by Proposition \ref{theorem-bifurcation-first}. 
If either $\langle g^2, L_1^{-1} g^2 \rangle < 0$ or $\langle g^2, L_1^{-1} g^2 \rangle > 0$, the bifurcating branch is a local minimizer of the constrained energy $H$ for fixed momentum $P$ and mass $Q$ degenerate only by the translational and rotational symmetries. 
\end{proposition}

\begin{proof}
We use again the decomposition (\ref{decomposition-bif}) with $\Omega = \Omega_c + \delta \Omega(a)$, where $(\mathfrak{w}_1(a), \mathfrak{z}_1(a)) \in H^2_{\rm even}(\mathbb{R}) \times H^2_{\rm even}(\mathbb{R})$ satisfy the bounds (\ref{bounds-on-terms}) 
and $\delta \Omega(a)$ is given by the expansion (\ref{delta-Omega}). Hence, we 
compute 
\begin{align}
\left\{ \begin{array}{l}
    P(U,0) = \frac{1}{2} \| U_0 \|^2_{L^2} + \langle U_0, \mathfrak{w}_1(a) \rangle + \frac{1}{2} \| \mathfrak{w}_1(a) \|^2_{L^2}, \\
    Q(A) = \frac{s}{k} \left( a^2 \| g \|^2_{L^2} + 2 a \langle g, \mathfrak{z}_1(a) \rangle + \| \mathfrak{z}_1(a) \|^2_{L^2} \right).
    \end{array} \right.
    \label{P-Q-expansion}
\end{align}
We recall the action functional (\ref{action}) rewritten as 
$$
\Lambda(U,\Psi) = H(U,A) + c P(U,0) + |\Omega| Q(A)
$$
To incorporate the constraints of fixed momentum $P$ and mass $Q$ in the minimization of energy $H$, we consider the Hessian operator $\mathcal{L}$
on the constrained subspace of $(L^2(\mathbb{R}))^3$ 
for the perturbation $(w,z_1,z_2)$ satisfying the following two constraints:
$$
\langle U, w \rangle = 0 \quad \mbox{\rm and} \quad \langle A, z_1 \rangle = 0,
$$
where $w$ is the perturbation to $U$ and $z_1 + i z_2$ is the perturbation to $A$ according to (\ref{variables-Psi-z}).
The constrained Hessian operator is denoted by \(\hat{\mathcal{L}}\). 
By Theorem 3.2 in \cite{Dmitry}, the Morse index $n(\hat{\mathcal{L}})$
and the nullity index $z(\hat{\mathcal{L}})$ are computed as 
\begin{align}
\left\{ \begin{array}{l}
n(\hat{\mathcal{L}}) = n(\mathcal{L})-p_0-z_0, \\
z(\hat{\mathcal{L}}) = z(\mathcal{L}) + z_0,
    \end{array} \right.
\label{index-count}
\end{align}
where $p_0$ and $z_0$ are the number of positive and zero  eigenvalues of the matrix ${\bf D}$ given by 
\begin{align*}
{\bf D} = 
\begin{pmatrix}
    \frac{\partial P(U,0)}{\partial c} & \frac{\partial P(U,0)}{\partial |\Omega|} \\
    \frac{\partial Q(A)}{\partial c} & \frac{\partial Q(A)}{\partial |\Omega|}   
\end{pmatrix}
\end{align*}
For the bifurcating branch, we use parameters $c$ and $a^2$ with the dependence $\Omega = \Omega_c + \delta \Omega(a)$ 
expressed in even powers of the small amplitude $a$. Hence, we introduce the function $(c,\Omega) \to a^2$ by solving the 
implicit equation $\Omega = \Omega_c + \delta \Omega(a)$ for $\Omega$ close to $\Omega_c$. By using the chain rule, 
we obtain 
\begin{align*}
{\bf D} = \begin{pmatrix}
    \frac{\partial P(U,0)}{\partial c} + \frac{\partial P(U,0)}{\partial a^2} \frac{\partial a^2}{\partial c} 
    & -\frac{\partial P(U,0)}{\partial a^2} \frac{\partial a^2}{\partial \Omega} \\
    \frac{\partial Q(A)}{\partial c} + \frac{\partial Q(A)}{\partial a^2} \frac{\partial a^2}{\partial c} 
    & -\frac{\partial Q(A)}{\partial a^2} \frac{\partial a^2}{\partial \Omega},  
\end{pmatrix}
\end{align*}
It follows from (\ref{delta-Omega}) that 
\begin{align*}
\frac{\partial a^2}{\partial c} &= \frac{\| g \|^2_{L^2}}{s k \langle g^2, L_1^{-1} g^2 \rangle} \frac{d \Omega_c}{d c} + \mathcal{O}(a^2), \\
\frac{\partial a^2}{\partial \Omega} &= -\frac{\| g \|^2_{L^2}}{s k \langle g^2, L_1^{-1} g^2 \rangle} + \mathcal{O}(a^2),
\end{align*}
where $\langle g^2, L_1^{-1} g^2 \rangle \neq 0$ is assumed. 
By using (\ref{momentum-zero}) and (\ref{near-identity}) in (\ref{P-Q-expansion}), we obtain 
\begin{align*}
 \frac{\partial P(U,0)}{\partial c} &= 18 \sqrt{c} + \mathcal{O}(a^2), \\
  \frac{\partial P(U,0)}{\partial a^2} &= \langle U_0, w_2 \rangle + \mathcal{O}(a^2), \\
  \frac{\partial Q(A)}{\partial c} &= \mathcal{O}(a^2), \\
  \frac{\partial Q(A)}{\partial a^2} & = \frac{s}{k} \| g \|^2_{L^2} + \mathcal{O}(a^2).
\end{align*}
Computing $\det {\bf D}$ yields a simpler formula 
\begin{align*}
    \det {\bf D} &= \frac{\partial a^2}{\partial \Omega} \left[ \frac{\partial P(U,0)}{\partial a^2} \frac{\partial Q(A)}{\partial c} - \frac{\partial P(U,0)}{\partial c} \frac{\partial Q(A)}{\partial a^2} \right] \\
    &= \frac{18 \sqrt{c} \| g \|^4_{L^2}}{k^2 \langle g^2, L_1^{-1} g^2 \rangle} + \mathcal{O}(a^2).
\end{align*}

\underline{If $\langle g^2, L_1^{-1} g^2 \rangle < 0$,} then $\det {\bf D} < 0$ for small $a$ and ${\bf D}$ has one positive and one negative eigenvalues. 
By Proposition \ref{theorem-Morse-first}, we have $z(\mathcal{L}) = 2$ and $n(\mathcal{L}) = 1$ in this case. Since $z_0 = 0$ and $p_0 = n(\mathcal{L}) = 1$, the count (\ref{index-count}) implies that 
the bifurcating branch is a local minimizer of the constrained energy $H$ for fixed momentum $P$ and mass $Q$. 

\underline{If $\langle g^2, L_1^{-1} g^2 \rangle > 0$,} then $\det {\bf D} > 0$ for small $a$ and since the second diagonal entry is positive, 
\begin{align*}
    -\frac{\partial Q(A)}{\partial a^2} \frac{\partial a^2}{\partial \Omega} = \frac{\| g \|^4_{L^2}}{k^2 \langle g^2, L_1^{-1} g^2 \rangle} + \mathcal{O}(a^2) > 0,
\end{align*}
${\bf D}$ has two positive eigenvalues. By Proposition \ref{theorem-Morse-first}, we have $z(\mathcal{L}) = 2$ and $n(\mathcal{L}) = 2$ in this case. 
Since $z_0 = 0$ and $p_0 = n(\mathcal{L}) = 2$, the count (\ref{index-count}) implies that the bifurcating branch is a local minimizer of the constrained energy $H$ for fixed momentum $P$ and mass $Q$. 

In either case, we have $z(\hat{\mathcal{L}}) = z(\mathcal{L}) = 2$, 
which implies that the local minimizer of the constrained energy is only degenerate due to the translational and rotational 
symmetries of the coupled KdV--LS system (\ref{KdV-NLS}). 
\end{proof}

\begin{remark}
We have two different cases of the pitchfork bifurcation.
	\begin{itemize}
		\item If $\langle g^2, L_1^{-1} g^2 \rangle < 0$, the bifurcating branch exists for $\Omega > \Omega_c$ and it inherits 
the minimizer property of the constrained energy from the branch of uncoupled KdV solitons, which becomes a saddle point of the constrained energy for $\Omega > \Omega_c$ by Proposition \ref{theorem-uncoupled}. 
This bifurcation is classified as the subcritical pitchfork bifurcation.

\item If $\langle g^2, L_1^{-1} g^2 \rangle > 0$, the bifurcating branch exists for $\Omega < \Omega_c$, where the branch of uncoupled KdV solitons is also a minimizer of the constrained energy by Proposition \ref{theorem-uncoupled}. This bifurcation can be classified as the supercritical pitchfork bifurcation, 
but it differs from the standard supercritical pitchfork bifurcation in the absence of symmetries, where the bifurcating branch is usually unstable. 
	\end{itemize}
Examples of pitchfork bifurcations where both the primary and bifurcating branches can be stable or unstable for the same parameters were given 
for a generalized NLS equation in \cite{Yang}.
\label{rem-bifucations}
\end{remark}

\begin{remark}
By using the exact solution (\ref{solution-2}) for $k = \frac{1}{6}$ and $s = 1$, we compute 
\begin{align*} 
P(U,0) &= 96 |\Omega|^{3/2}, \\ 
Q(A) &= 144(c - 4|\Omega|) |\Omega|^{1/2}.
\end{align*}
This yields the expression for ${\bf D}$ in 
\begin{align}
\label{D-exact}
{\bf D} = 
\begin{pmatrix}
    \frac{\partial P(U,0)}{\partial c} & \frac{\partial P(U,0)}{\partial |\Omega|} \\
    \frac{\partial Q(A)}{\partial c} & \frac{\partial Q(A)}{\partial |\Omega|}   
\end{pmatrix}
= \begin{pmatrix}
    0 & 144 |\Omega|^{1/2} \\
    144 |\Omega|^{1/2} & 72 |\Omega|^{-1/2} (c - 16 |\Omega|)  
\end{pmatrix}
\end{align}
Since $\det {\bf D} < 0$, ${\bf D}$ has one positive and one negative eigenvalue. At the same time, Morse index for the exact solution (\ref{solution-2}) is equal to $1$, see Lemma \ref{lem-exact-2} below, hence the computations in the proof of Proposition \ref{theorem-stability-first} are in agreement with (\ref{D-exact}). Furthermore, since 
$$
k = \frac{1}{6} : \quad \Omega = -\frac{c}{4} + \frac{a^2}{12 c} + \mathcal{O}(a^4)
$$
and 
$$
\langle U_0, w_2 \rangle = \frac{24}{\sqrt{c}} \int_{\mathbb{R}} W(y) {\rm sech}^2(y) dy = -\frac{6}{\sqrt{c}},
$$
where $W$ is given by (\ref{W-expression}), we obtain 
\begin{align*}
\frac{\partial P(U,0)}{\partial c} + \frac{\partial P(U,0)}{\partial a^2} \frac{\partial a^2}{\partial c} &= 
18 \sqrt{c} + 3 c \langle U_0, w_2 \rangle + \mathcal{O}(a^2) = \mathcal{O}(a^2),
\end{align*}
in agreement with the first diagonal term of ${\bf D}$ in (\ref{D-exact}) being zero.
\label{remark-exact-2}
\end{remark}

\subsection{The Hessian operator for the exact solution (\ref{solution-2})}

Here, we verify that the Hessian operator $\mathcal{L}$ 
at the exact solution (\ref{solution-2}) for $k = \frac{1}{6}$ and $s = 1$ has 
a simple negative eigenvalue and a double zero eigenvalue for the entire existence 
interval $\Omega \in (\Omega_c,0)$, where $\Omega_c = -\frac{c}{4}$. Due to the exact computation 
of ${\bf D}$ in (\ref{D-exact}), see Remark \ref{remark-exact-2}, 
this suggests that the exact solution (\ref{solution-2}) is a local 
minimizer of the constrained energy $H$ for fixed momentum $P$ and mass $Q$ degenerate only by the translational and rotational symmetries. This result 
recovers the orbital stability proven in \cite{Chen} and extends the results of Propositions \ref{theorem-Morse-first} and \ref{theorem-stability-first} beyond the local 
bifurcation limit along the exact solution (\ref{solution-2}).

Let $k = \frac{1}{6}$ and $s = 1$. By using (\ref{solution-2}), we compute the block $\mathcal{L}_J$ of the Hessian operator $\mathcal{L}$ given by (\ref{block-1}) explicitly:
\begin{equation*}
\mathcal{L}_J = \begin{pmatrix} -\partial_{\xi}^2 + c - 12 |\Omega| {\rm sech}^2(|\Omega|^{1/2} \xi)  & -4 \sqrt{3|\Omega|(c - 4 |\Omega|)} {\rm sech}(|\Omega|^{1/2} \xi)  \\ -4 \sqrt{3|\Omega|(c - 4 |\Omega|)} {\rm sech}(|\Omega|^{1/2} \xi) & 
12 \left( - \partial_{\xi}^2 + |\Omega| - 2 |\Omega| {\rm sech}^2(|\Omega|^{1/2} \xi) \right).
\end{pmatrix}
\end{equation*}
The exact solution (\ref{solution-2}) exists for $\Omega \in \left( -\frac{c}{4},0\right)$. With a change of variables 
$$
\eta = \sqrt{|\Omega|} \xi, \quad \gamma = \frac{c}{|\Omega|}, 
$$
we get $\mathcal{L}_J = |\Omega| L_J(\gamma)$, where $L_J(\gamma)$ is given by 
\begin{equation}
    \label{Hessian-exact-normalized}
L_J(\gamma) = \begin{pmatrix} -\partial_{\eta}^2 + \gamma - 12 {\rm sech}^2(\eta)  & -4 \sqrt{3 (\gamma - 4)} {\rm sech}(\eta)  \\ -4 \sqrt{3 (\gamma - 4)} {\rm sech}(\eta)  & 
12 \left( - \partial_{\eta}^2 + 1 - 2 {\rm sech}^2(\eta) \right).
\end{pmatrix}
\end{equation}
The only parameter $\gamma > 4$ parametrizes $L_J(\gamma)$ and the result of Proposition \ref{theorem-Morse-first} suggests 
that for small $|\gamma - 4|$, $L_J(\gamma)$ has a simple negative eigenvalue near $-5$, a simple zero eigenvalue, and at least two positive 
eigenvalue (one is close to $0$ and the other one is close to $3$) below the essential spectrum 
on $[\gamma,\infty) \cup [12,\infty)$. 
The following lemma guarantees the persistence of this result for every $\gamma > 4$. 



\begin{lemma}
    \label{lem-exact-2}
The linear operator $L_J(\gamma)$ admits simple negative and zero eigenvalues for every $\gamma > 4$.
\end{lemma}

\begin{proof}
We recall that $L_J(\gamma)$ in (\ref{Hessian-exact-normalized}) admits a zero eigenvalue for every $\gamma > 4$ due to the translational mode 
\begin{equation}
    \label{translation}
\mathcal{L}_J \begin{pmatrix} U' \\ A' \end{pmatrix} = \begin{pmatrix} 0 \\ 0 \end{pmatrix},
\end{equation}
which exists for every $\Omega \in \left( -\frac{c}{4},0\right)$. To prove the assertion, we show that the zero eigenvalue of $L_J(\gamma)$ remains simple for all $\gamma > 4$. This implies by continuity of eigenvalues in $\gamma$ that 
$\mathcal{L}_J(\gamma)$ admits a simple negative eigenvalue for every $\gamma > 4$ since this is true for small $\gamma \gtrsim 4$. 

By looking at the zero eigenvalue of $L_J(\gamma)$ with the eigenvector $(w,z)^{\perp}$, we take  $w = 4 \sqrt{3(\gamma - 4)} \upsilon$ and rewrite the spectral problem in the equivalent form:
\begin{equation}
\label{system-closed}
    \left\{ \begin{array}{l} L_1(\gamma) \upsilon = {\rm sech}(\eta) z, \\
    L_2 z = 4(\gamma - 4) {\rm sech}(\eta) v. \end{array} \right.
\end{equation}
where $L_1(\gamma) = -\partial_{\eta}^2 + \gamma - 12 {\rm sech}^2(\eta)$ and $L_2 = -\partial_{\eta}^2 + 1 - 2 {\rm sech}^2(\eta)$ 
are two self-adjoint Schr\"{o}dinger operators in $L^2(\mathbb{R})$ with the domains in $H^2(\mathbb{R})$. Due to the translational mode 
(\ref{translation}), we can find one decaying solution of (\ref{system-closed}) in the form
\begin{equation}
    \label{translation-0}
\upsilon_0 = \tanh(\eta) {\rm sech}^2(\eta), \quad z_0 =  (\gamma - 4) \tanh(\eta) {\rm sech}(\eta).
\end{equation}

The operator $L_1(\gamma) : H^2(\mathbb{R}) \subset L^2(\mathbb{R}) \to L^2(\mathbb{R})$ is invertible for every $\gamma \in (4,9) \cup (9,\infty)$. 
This implies that for every $z \in L^2(\mathbb{R})$, there exists a unique $\upsilon = (L_1(\gamma))^{-1} {\rm sech}(\eta) z$ in $H^2(\mathbb{R})$ 
so that the system (\ref{system-closed}) can be rewritten as the generalized eigenvalue problem 
\begin{equation}
\label{generalized-eig}
    L_2 z = \mu \mathcal{K} z, \quad \mathcal{K} = {\rm sech}(\eta) (L_1(\gamma))^{-1} {\rm sech}(\eta), \quad \mu = 4(\gamma - 4).
\end{equation}
The linear operator $\mathcal{K} : L^2(\mathbb{R}) \to L^2(\mathbb{R})$ is self-adjoint and compact due to the exponential decay 
of ${\rm sech}(\eta)$ as $|\eta| \to \infty$ \cite[Section 5.6]{Gustafson}. Furthermore, $L_2 : H^2(\mathbb{R}) \subset L^2(\mathbb{R}) \to L^2(\mathbb{R})$ 
is non-negative with a simple eigenvalue at $0$ for the eigenfunction ${\rm sech}(\eta)$. Consequently, the generalized eigenvalue 
problem (\ref{generalized-eig}) admits the zero eigenvalue $\mu = 0$ with the eigenfunction ${\rm sech}(\eta)$. We claim that $\mu = 0$ is algebraically simple for every $\gamma \in (4,9) \cup (9,\infty)$. By the Fredholm theorem, the simplicity of $\mu = 0$ implies 
that 
\begin{equation}
    \label{Fredholm-cond}
\langle \mathcal{K} {\rm sech}(\eta), {\rm sech}(\eta) \rangle \neq 0, 
\end{equation}
where
\begin{equation*}
\langle \mathcal{K} {\rm sech}(\eta), {\rm sech}(\eta) \rangle = \langle (L_1(\gamma))^{-1} {\rm sech}^2(\eta), {\rm sech}^2(\eta) \rangle.
\end{equation*}
Indeed, if $\gamma = 4$, we have the explicit formula, see (\ref{expression-W}) below,
$$
(L_1(4))^{-1} {\rm sech}^2(\eta) = \frac{1}{4} ( \eta \tanh(\eta) - 1) {\rm sech}^2(\eta),
$$
which yields 
$$
\langle (L_1(4))^{-1} {\rm sech}^2(\eta), {\rm sech}^2(\eta) \rangle = -\frac{1}{4}. 
$$
Because 
$$
\frac{d}{d\gamma} (L_1(\gamma))^{-1} = -(L_1(\gamma))^{-1} \frac{d}{d\gamma} L_1(\gamma) (L_1(\gamma))^{-1} = -(L_1(\gamma))^{-1} (L_1(\gamma))^{-1},
$$
it follows that 
\begin{align*}
\frac{d}{d \gamma} \langle (L_1(\gamma))^{-1} {\rm sech}^2(\eta), {\rm sech}^2(\eta) \rangle &= 
-\langle (L_1(\gamma))^{-1} (L_1(\gamma))^{-1} {\rm sech}^2(\eta), {\rm sech}^2(\eta) \rangle \\
&= -\| (L_1(\gamma))^{-1} {\rm sech}^2(\eta) \|^2 < 0.
\end{align*}
Hence, $\langle \mathcal{K} {\rm sech}(\eta), {\rm sech}(\eta) \rangle$ is monotonically decreasing for $\gamma \in (4,9)$ and remains negative. 
Since $L_1(9) {\rm sech}^3(\eta) = 0$, it is clear that $\lim_{\gamma \to 9^{\pm}} \langle \mathcal{K} {\rm sech}(\eta), {\rm sech}(\eta) \rangle = \pm \infty$. 
Then, $\langle \mathcal{K} {\rm sech}(\eta), {\rm sech}(\eta) \rangle$ is monotonically decreasing for $\gamma \in (9,\infty)$ and remains positive 
because $L_1(\gamma)$ is strictly positive for $\gamma > 9$. Thus, the Fredholm condition (\ref{Fredholm-cond}) 
for simplicity of $\mu = 0$ is proven for every $\gamma \in (4,9) \cup (9,\infty)$. 

Under the condition (\ref{Fredholm-cond}),  the two-term decomposition 
$$
z = z_{\perp} - \frac{\langle \mathcal{K} z_{\perp}, {\rm sech}(\eta) \rangle}{\langle \mathcal{K} {\rm sech}(\eta), {\rm sech}(\eta) \rangle} {\rm sech}(\eta), \quad \langle z_{\perp}, {\rm sech}(\eta) \rangle = 0
$$
reduces (\ref{generalized-eig}) to the form 
\begin{equation}
\label{generalized-eig-projected}
\Pi_0 L_2 \Pi_0 z_{\perp} = \mu \mathcal{K}_{\perp} z_{\perp}, \quad \mathcal{K}_{\perp} z_{\perp} = \mathcal{K} z_{\perp} - \frac{\langle \mathcal{K} z_{\perp}, {\rm sech}(\eta) \rangle}{\langle \mathcal{K} {\rm sech}(\eta), {\rm sech}(\eta) \rangle } \mathcal{K} {\rm sech}(\eta) 
\end{equation}
where $\Pi_0$ is an orthogonal projection from $L^2(\mathbb{R})$ to $L^2(\mathbb{R}) |_{\{ {\rm sech}(\eta)\}^{\perp}}$. 
Since $\Pi_0 L_2 \Pi_0$ is self-adjoint, strictly positive, and invertible with a bounded inverse in $L^2(\mathbb{R})$ 
and $\mathcal{K}_{\perp}$ is self-adjoint and compact in $L^2(\mathbb{R})$, the spectrum of the generalized eigenvalue problem 
(\ref{generalized-eig-projected}) is purely discrete and consists of real eigenvalues $\mu \neq 0$. 
Due to the translational mode (\ref{translation-0}), the set of eigenvalues includes the eigenvalue $\mu_0(\gamma) = 4 (\gamma - 4)$ 
with the eigenfunction $z_0 = (\gamma - 4) \tanh(\eta) {\rm sech}(\eta)$, for which $\mathcal{K}_{\perp} z_0 = \mathcal{K} z_0$.

Each eigenvalue $\mu$ of the generalized eigenvalue problem (\ref{generalized-eig}) is geometrically simple, 
because there is at most one linearly independent eigenfunction $z \in H^2(\mathbb{R})$ convergent to a one-dimensional 
subspace spanned $e^{-|\eta|}$ as $|\eta| \to \infty$ due to the exponential decay of ${\rm sech}(\eta)$ in $\mathcal{K}$. 
Moreover, each eigenvalue $\mu \neq 0$ of (\ref{generalized-eig}) with the eigenfunction $z \in H^2(\mathbb{R})$ is algebraically simple since 
$$
\langle \mathcal{K} z, z \rangle = \frac{1}{\mu} \langle L_2 z, z \rangle \neq 0
$$
since $L_2$ is non-negative and $z$ for $\mu \neq 0$ cannot be spanned by ${\rm sech}(\eta)$, the eigenfunction of $L_2$. 
This can be checked explicitly for $\mu_0(\gamma) = 4 (\gamma - 4)$ by using 
(\ref{translation-0}):
\begin{equation}
    \label{conclusion}
\langle \mathcal{K} z_0, z_0 \rangle = \langle \upsilon_0, {\rm sech}(\eta) z_0 \rangle =  (\gamma - 4) \int_{\mathbb{R}} \tanh^2(\eta) {\rm sech}^4(\eta) d\eta 
= \frac{4}{15} (\gamma - 4) > 0.
\end{equation}
Hence, the eigenvalue $\mu_0(\gamma) = 4 (\gamma - 4)$ is algebraically simple for every $\gamma \in (4,9) \cup (9,\infty)$.

To include the exceptional case $\gamma = 9$, we note that $L_1(9) {\rm sech}^3(\eta) = 0$ and $L_2 {\rm sech}(\eta) =0$. By using the decomposition 
$$
\upsilon = a {\rm sech}^3(\eta) + \upsilon_{\perp}, \quad z = b {\rm sech}(\eta) + z_{\perp}, \quad \langle {\rm sech}^3(\eta), \upsilon_{\perp} \rangle = \langle {\rm sech}(\eta), z_{\perp} \rangle = 0,
$$
we obtain $a$ and $b$ uniquely from projections of the right-hand sides of the system (\ref{system-closed}) to the orthogonal complements 
of the kernels of $L_1$ and $L_2$ to obtain 
$$
a = -\frac{\langle {\rm sech}^2(\eta),w_{\perp} \rangle}{\langle {\rm sech}^2(\eta), {\rm sech}^3(\eta) \rangle}, \quad 
b = -\frac{\langle {\rm sech}^4(\eta),z_{\perp} \rangle}{\langle {\rm sech}^4(\eta), {\rm sech}(\eta) \rangle}.
$$
Inverting $L_1(9)$ on the orthogonal subspace spanned by ${\rm sech}^3(\eta)$, we obtain 
$$
v_{\perp} = \Pi_1 (L_1(9))^{-1} \Pi_1 {\rm sech}(\eta) z_{\perp},
$$ 
where $\Pi_1$ is an orthogonal projection from $L^2(\mathbb{R})$ to $L^2(\mathbb{R}) |_{\{ {\rm sech}^3(\eta) \}^{\perp}}$. Substituting $v_{\perp}$ to the equation for $z_{\perp}$ yields a generalized eigenvalue problem 
\begin{equation*}
    \Pi_0 L_2 \Pi_0 z_{\perp} = \mu \mathcal{K}_{\perp} z_{\perp}, 
\end{equation*}
with $\mu = 4(\gamma - 4)$ and 
\begin{equation*}
\mathcal{K}_{\perp} = \Pi_0 {\rm sech}(\eta) \Pi_1 (L_1(9))^{-1} \Pi_1 {\rm sech}(\eta) \Pi_0.
\end{equation*}
The spectrum of the generalized eigenvalue problem is again purely discrete and each eigenvalue is geometrically simple. For 
the eigenvalue $\mu = 4 (\gamma - 4) |_{\gamma = 9} = 20$, the eigenfunction $z_0$ is 
odd in $\eta$, hence $a = b = 0$ and $\langle \mathcal{K}_{\perp} z_0, z_0 \rangle = \langle \mathcal{K} z_0, z_0 \rangle \neq 0$ as in (\ref{conclusion}). Hence, $\mu = 20$ remains algebraically simple at $\gamma = 9$. 
\end{proof}

\subsection{Numerical approximation of $\langle g^2, L_1^{-1} g^2 \rangle$}

Here, we check numerically the value of  $\langle g^2, L_1^{-1} g^2 \rangle$, which appears in Proposition \ref{theorem-bifurcation-first} to separate the 
supercritical and subcritical pitchfork bifurcations, see Remark \ref{rem-bifucations}.

To compute $\langle g^2, L_1^{-1} g^2 \rangle$, we use the scaling transformation 
$$
L_1^{-1} g^2 = \frac{4}{c} W(y), \quad y = \frac{\sqrt{c}}{2} \xi, 
$$ 
and obtain $W(y)$ from solutions of the linear inhomogeneous equation 
\begin{equation}
\label{eq-w-2}
-W'' + 4 W -  12 \text{sech}^2(y) W = \text{sech}^{2p}(y).
\end{equation}
Two solutions of the homogeneous equations are given by 
$$
W_1 =  \tanh(y) {\rm sech}^2(y), \quad 
W_2=\frac{5}{8}+\frac{1}{4}\text{cosh}^2(y)+\frac{15}{8}{\rm sech}^2(y)(y  {\rm tanh}(y)-1)
$$
which are computed due to the symmetry mode $U_0'$ being in the kernel of $L_1$, see also (\ref{translation}). 
The expression for $W_2$ has been normalized with the Wronskian of $W_1$ and $W_2$ being equal to $1$. 
We obtain by using the variation of parameter formula that 
\begin{equation}
\label{W-expression}
W(y) = W_1(y) \int_{y_0}^{y}W_2g^2dy - W_2(y) \int_{y_1}^{y}W_1g^2dy,
\end{equation}
where $y_0$ and $y_1$ need to be defined from the condition that $W \in H^2_{\rm even}(\mathbb{R})$. 

For the first term in (\ref{W-expression}), since \(W_2 g^2\) is even and $W_1$ is odd, we can set $y_0 = 0$ to ensure that 
$\int_{0}^{y}W_2g^2dy$ is odd and $W_1 \int_{0}^{y}W_2g^2dy$ is even. Moreover, 
for $\vert y\vert \gg 1$, we have $W_1(y) \sim e^{-2\vert y\vert }$, $W_2(y) \sim e^{2\vert y\vert}$, and $g^2(y)\sim e^{-2p\vert y\vert }$ 
so that $W_1(y) \int_0^y W_2g^2 dy \sim e^{-2p\vert y\vert }$ decays to $0$ exponentially. 

For the second term in (\ref{W-expression}), we note that $\int_{-\infty}^{\infty} W_1 g^2 dy = 0$ since $W_1$ is odd and $g^2$ is even. 
Hence, we set $y_1 = -\infty$, which yields the exponential decay of $W_2(y) \int_{-\infty}^y W_1g^2 dy \sim e^{-2p\vert y\vert}$ 
to $0$ for $\vert y\vert \gg 1$. Moreover, we show that $G(y) := W_2(y) \int_{-\infty}^y W_1g^2 dy$ is even since 
\begin{align*} 
G(-y) &= W_2(-y) \int_{-\infty}^{-y} W_1(y') g^2(y') dy' = W_2(y) \int_y^{+\infty} W_1(-y') g^2(-y') d y' \\
&= - W_2(y) \int_y^{+\infty} W_1(y') g^2(y') d y' = W_2(y) \int_{-\infty}^y W_1(y') g^2(y') d y'= G(y).
\end{align*}
Thus, $W$ in (\ref{W-expression}) is even if $y_0 = 0$ and $y_1 = -\infty$. Moreover, $W(y) \to 0$ as $|y| \to \infty$ exponentially 
so that $W \in H^2_{\rm even}(\mathbb{R})$. 

\begin{figure}[htp!]
	\centering
	\includegraphics[width=0.55\textwidth]{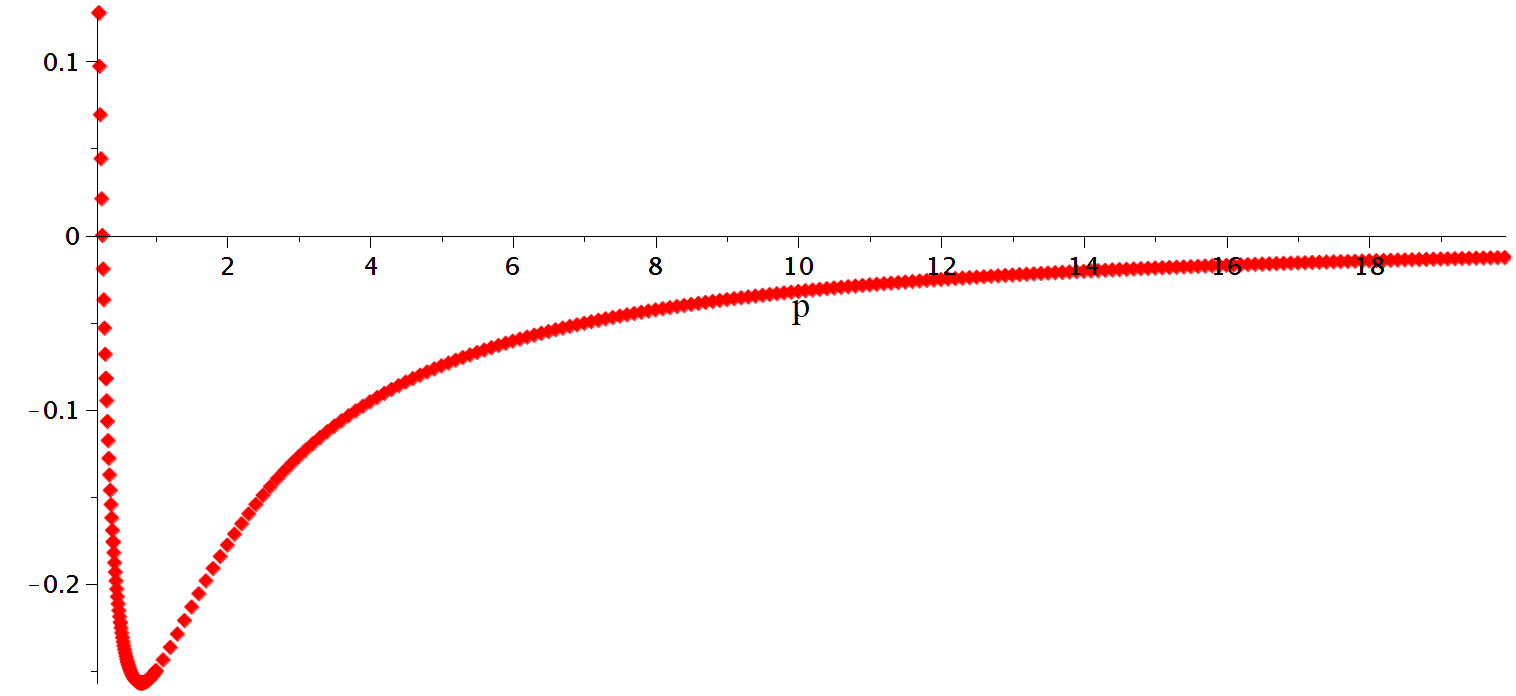}
	\caption{Numerical approximation of $\int_{\mathbb{R}} g^2 W dy$ versus $p$.}
	\label{fig-projection}
\end{figure}

The expression (\ref{W-expression}) was used for the numerical calculation 
of 
$$
\int_{\mathbb{R}} g^2 W dy = \frac{\sqrt{c^3}}{8} \langle g^2, L_1^{-1} g^2 \rangle,
$$
which is plotted in Figure \ref{fig-projection} versus $p$. It follows from Figure \ref{fig-projection} that the value of $\langle g^2, L_1^{-1} g^2 \rangle$ is positive for very small values of 
$p$ and negative for larger values of $p$. Due to (\ref{g-eigenfunction}), the values of $p$ are proportional to the values of $k$, where $s = {\rm sgn}(k) = +1$ implies that $k > 0$. In the special case \(k=\frac{1}{6}\), we have \(p=1\), for which we obtain the exact solution of (\ref{eq-w-2}):
\begin{equation}
\label{expression-W}
W(y)=\frac{1}{4} (y\text{tanh}(y)-1) \text{sech}^2(y).
\end{equation}
Computing the integral yields $\int_{\mathbb{R}} g^2 W dy = -\frac{1}{4}$ for $p = 1$, in agreement with Figure \ref{fig-projection}. 
The value $p = 1$ correspond to the minimum of $\langle g^2, L_1^{-1} g^2 \rangle$ with respect to $p$.

\section{Bifurcation of the coupled solitary wave at $\Omega_c^{(2)}$}
\label{sec-5}

\subsection{The bifurcating branch of coupled solitary waves} 

By Proposition \ref{theorem-uncoupled}, the second bifurcation along the family of the uncoupled KdV solitons occurs at $\Omega = \Omega_c^{(2)}$, see (\ref{Omega-bif-intro}), which we denote as 
\begin{equation}
\label{second-bifurcation} 
\widetilde{\Omega}_c=-\frac{c}{16}(\sqrt{1+48 k}-3)^2, \quad k > \frac{1}{6}.
\end{equation}
The restriction $k > \frac{1}{6}$ is due to the existence of the second isolated eigenvalue of the Schr\"{o}dinger operator $L_2$, see (\ref{eigenvalues-Schr}) and (\ref{block-2}). The second eigenvalue of $L_2$ corresponds to the following eigenfunction:
\begin{equation}
\label{g_tilde-eigenfunction}
    \tilde{g}(\xi) = \text{sech}^{q}\left( \frac{\sqrt{c}}{2} \xi\right)\text{tanh}\left( \frac{\sqrt{c}}{2} \xi\right), \quad q = \frac{\sqrt{1+48 k}-3}{2}.
\end{equation}

\begin{figure}[htb!]
	\centering
	\includegraphics[width=0.55\textwidth]{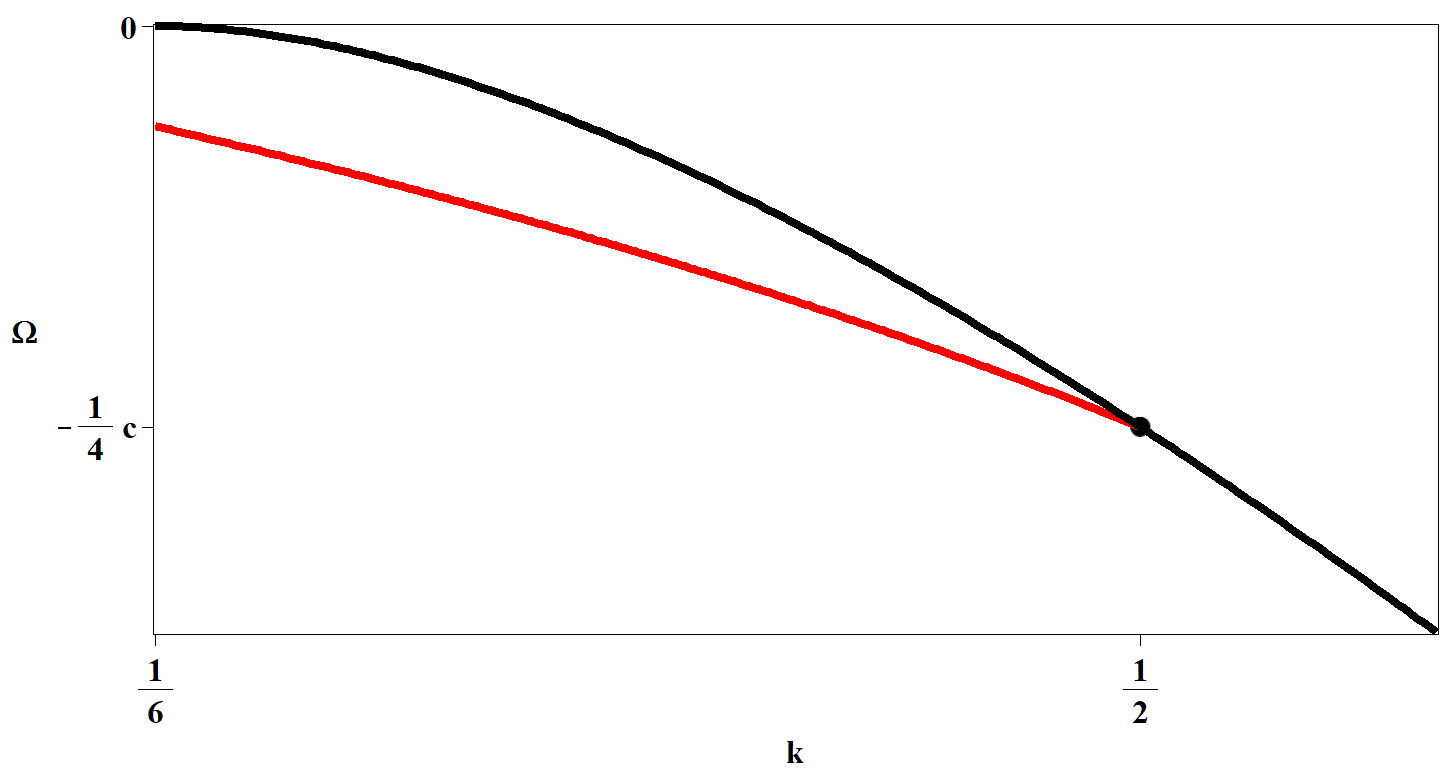}
	\caption{The existence curve \(\Omega_{\rm exact}(k,c)\)  in (\ref{Omega-third})  (red) and the bifircation curve 
		\(\widetilde{\Omega}_c\) (black) versus $k \in \left(\frac{1}{6},\frac{1}{2}\right)$ for $c = 1$.}
	\label{fig-region-third}
\end{figure}

Since we require $s = {\rm sgn}(k) = +1$, the exact solution (\ref{solution-3}) is recovered from this bifurcation if $k = \frac{1}{2}$ for which $q = 1$ and 
$\widetilde{\Omega}_c = -\frac{c}{4}$, so that the profiles $U$ in (\ref{solution-1}) and $\tilde{g}$ in (\ref{g_tilde-eigenfunction}) coincide
with the profile $(U,A)$ of the exact solution (\ref{solution-3}) for $\Omega = -\frac{c}{4}$. The exact solution (\ref{solution-3}) exists 
for $k \in \left(0,\frac{1}{2}\right)$ at the specific value of $\Omega = \Omega_{\rm exact}(k,c)$ given by 
\begin{equation}
    \label{Omega-third}
\Omega_{\rm exact}(k,c) = -\frac{c k}{3 - 2 k}. 
\end{equation}
Because $\Omega$ is required to be in $\left(-\frac{c}{4},0\right)$, the exact solution (\ref{solution-3}) does not exist if $k > \frac{1}{2}$. 
The dependence of $\widetilde{\Omega}_c$ in (\ref{second-bifurcation}) and $\Omega_{\rm exact}(k,c)$ in (\ref{Omega-third}) 
is shown versus $k$ for fixed $c = 1$ in Figure \ref{fig-region-third} by black and red color, respectively. 
The exact solution (\ref{solution-3}) exists for $\Omega_{\rm exact}(k,c) < \widetilde{\Omega}_c$, which suggests the subcritical pitchfork bifurcation  
at $\widetilde{\Omega}_c$ for $k \in \left(\frac{1}{6},\frac{1}{2}\right)$.

The following result describes the bifurcating branch of coupled solitary waves at $\widetilde{\Omega}_c$. 

\begin{proposition}
    \label{theorem-bifurcation-second}    
    Assume $c > 0$, $\Omega < 0$, and $s = {\rm sgn}(k) = 1$.
    Let $U_0$ be given by (\ref{solution-1}) and $\tilde{g}$ be given by (\ref{g_tilde-eigenfunction}). 
    Assume that $\langle \tilde{g}^2, L_1^{-1} \tilde
    {g}^2 \rangle \neq 0$, where $L_1$ is defined in (\ref{block-2}). 
    For every $\Omega$ near $\tilde{\Omega}_c$ such that ${\rm sgn}(\Omega - \widetilde{\Omega}_c) = -{\rm sgn}(\langle \tilde{g}^2, L_1^{-1} \tilde{g}^2 \rangle)$, 
    there exists a unique family of solutions with the profile $(U,A) \in H^2_{\rm even}(\mathbb{R}) \times H^2_{\rm odd}(\mathbb{R})$ satisfying 
   $$
    \| U - U_0 \|_{H^2} \leq C |\Omega - \widetilde{\Omega}_c|, \quad \| A \| \leq C \sqrt{|\Omega - \widetilde{\Omega}_c|},
    $$
    for some $\Omega$-independent constant $C > 0$.
\end{proposition}

\begin{proof}
The proof is analogous to the proof of Proposition \ref{theorem-bifurcation-first} but the vector fields \(\mathbf{F}\) in (\ref{vector-field}) 
is now expanded near \(\widetilde{\Omega}_c\). The operator $L_2$ in (\ref{vector-perturbed}) is now 
defined by (\ref{block-2}) with $\Omega = \widetilde{\Omega}_c$ and $\delta \Omega = \Omega - \widetilde{\Omega}_c$. 
The decomposition (\ref{LSdecomposition}) is performed with $g$ replaced by $\tilde{g}$, from which we obtain 
the implicit equation 
\begin{equation}
    \label{fixed-point-4}
\begin{pmatrix} L_1 & 0  \\ 
0 & L_2 
\end{pmatrix} \begin{pmatrix} w_1 \\ z_1 \end{pmatrix} = \begin{pmatrix} \frac{1}{2} w_1^2 + s (a \tilde{g} +z_1)^2 \\  
    2 s w_1 (a \tilde{g} + z_1) + \frac{2 s}{k} \delta \Omega (a \tilde{g} + z_1) \end{pmatrix},
\end{equation}
subject to the solvability condition 
\begin{equation}
    \label{solv-eq_2}
\delta \Omega a \| \tilde{g} \|^2_{L^2} + k  \langle \tilde{g}, w_1 (a \tilde{g} + z_1)\rangle = 0. 
\end{equation}
Since $\ker L_1 = {\rm span}(U_0')$ is odd and $\ker L_2 = {\rm span}(\tilde{g})$ is odd in $\xi$, we consider 
now the implicit equation (\ref{fixed-point-4}) on the subspace of functions $(w_1,z_1) \in H^2_{\rm even}(\mathbb{R}) \times H^2_{\rm odd}(\mathbb{R})$, 
for which $w_1$ is even and $z_1$ is odd in $\xi$. The implicit equation (\ref{fixed-point-4}) is closed in this subspace 
thanks to the reversibility symmetry of the system (\ref{ode-system}), see Remark \ref{rem-reversibility}.
The solvability condition (\ref{solv-eq_2}) is needed because the second equation of the system (\ref{fixed-point-4}) 
contains $L_2$ acting on odd $z_1$ with odd $\ker L_2 = {\rm span}(\tilde{g})$.

By the implicit function theorem for the system (\ref{fixed-point-4}), we obtain the following solution for the bifurcating branch 
\[
U = U_0 + a^2 s L_1^{-1} \tilde{g}^2 + \mathfrak{w}_2(a), \quad 
A = a \tilde{g} + \mathfrak{z}_1(a),
\]
where $(\mathfrak{w}_2(a), \mathfrak{z}_1(a)) \in H^2_{\rm even}(\mathbb{R}) \times H^2_{\rm odd}(\mathbb{R})$ satisfy 
\begin{equation*}
\| \mathfrak{w}_2(a) \|_{H^2} \leq C a^4, \qquad \| \mathfrak{z}_1(a) \|_{H^2} \leq C |a|^3.
\end{equation*}
Expansion of the solvability condition (\ref{solv-eq_2}) then yields
\begin{equation*}
\delta \Omega = - k a^2  \frac{\langle \tilde{g}^2, L_1^{-1} \tilde{g}^2 \rangle}{\| \tilde{g} \|^2_{L^2}} + \mathcal{O}(a^4),
\end{equation*}
which shows that ${\rm sgn}(\delta \Omega) = -{\rm sgn}(\langle \tilde{g}^2, L_1^{-1} \tilde{g}^2 \rangle)$ since $k > 0$. This completes the proof.
\end{proof}

\subsection{Characterization of the bifurcating branch}

We first compute the Morse index and the nullity index of the bifurcating branch in Proposition \ref{theorem-bifurcation-second}. 

\begin{proposition}
    \label{theorem-Morse-second}
    Let $(U,A)$ be the profile of the bifurcating branch for $\Omega$ near $\widetilde{\Omega}_c$ given by Proposition \ref{theorem-bifurcation-second}. 
   Its Morse index is equal to $3$ if $\langle \tilde{g}^2, L_1^{-1} \tilde{g}^2 \rangle < 0$ and to $4$ if $\langle \tilde{g}^2, L_1^{-1} \tilde{g}^2 \rangle > 0$, whereas the nullity index is equal to $2$ in both cases. 
\end{proposition}

\begin{proof}
Following the same steps as in Proposition \ref{theorem-Morse-first}, we find that the double zero eigenvalue of the Hessian block $L_J$ split 
into a simple zero eigenvalue due to translational symmetry and a nonzero eigenvalue $\lambda = a \lambda^{(1)} + a^2 \lambda^{(2)} + \mathcal{O}(a^3)$ with 
\[\lambda^{(1)}=0, \ \ 
\lambda^{(2)}=-4\frac{\langle \tilde{g}^2,L_{1}^{-1}\tilde{g}^2\rangle}{\Vert \tilde{g}\Vert^2_{L^2}}.
\]
Due to the splitting of the zero eigenvalue of $L_J$ for small $a$, the Morse index of \(L_J\) is $2$ if \(\langle \tilde{g}^2,L_{1}^{-1}\tilde{g}^2\rangle<0\) and $3$ if \(\langle \tilde{g}^2,L_{1}^{-1}\tilde{g}^2\rangle>0\). In addition, the Schr\"{o}dinger 
operator \(L_2\) has a simple negative eigenvalue and a zero eigenvalue for 
small $a$. Referring back to the Hessian operator \(\mathcal{L}\) which is block-diagonalized into $L_J$ and $L_2$, we obtain the Morse index of \(\mathcal{L}\) being equal to $3$ if \(\langle \tilde{g}^2,L_{1}^{-1}\tilde{g}^2\rangle<0\) and to $4$ if \(\langle \tilde{g}^2,L_{1}^{-1}\tilde{g}^2\rangle>0\) for $\Omega \neq \widetilde{\Omega}_c$, whereas the nullity index of \(\mathcal{L}\) is $2$ for $\Omega \neq \widetilde{\Omega}_c$ due to the simple zero eigenvalue of $L_J$ and 
the simple zero eigenvalue of $L_2$.
\end{proof}

We next clarify the constrained minimization properties of the bifurcating branch by using information about the Morse and nullity indices from Proposition \ref{theorem-Morse-second}.

\begin{proposition}
    \label{theorem-stability-second}
    Let $(U,A)$ be the profile of the bifurcating branch for $\Omega$ near $\widetilde{\Omega}_c$ given by Proposition \ref{theorem-bifurcation-second}. 
If either $\langle \tilde{g}^2, L_1^{-1} \tilde{g}^2 \rangle < 0$ or $\langle \tilde{g}^2, L_1^{-1} \tilde{g}^2 \rangle > 0$, the bifurcating branch 
is a saddle point of the constrained energy $H$ for fixed momentum $P$ and mass $Q$ with exactly two negative eigendirections and degenerate only by the translational and rotational symmetries. 
\end{proposition}

\begin{proof}
We follow the same steps as in the proof of Proposition \ref{theorem-stability-first} and obtain the same determinant formula with \(g\) replaced with \(\tilde{g}\):
\[ 
\det {\bf D} 
    = \frac{18 \sqrt{c} \| \tilde{g} \|^4_{L^2}}{k^2 \langle \tilde{g}^2, L_1^{-1} \tilde{g}^2 \rangle} + \mathcal{O}(a^2).
\]

\underline{If $\langle \tilde{g}^2, L_1^{-1} \tilde{g}^2 \rangle < 0$,} then 
$\det {\bf D}  < 0$ and \(\mathbf{D}\) has one positive and one negative eigenvalue. Hence \(p_0=1\) and \(z_0=0\). 
Since $n(\mathcal{L}) = 3$ and $z(\mathcal{L}) = 2$ by Proposition  \ref{theorem-Morse-second}, the count (\ref{index-count}) implies 
\begin{align*}
\left\{ \begin{array}{l} 
n(\hat{\mathcal{L}}) = n(\mathcal{L})-p_0-z_0=3-1=2, \\
z(\hat{\mathcal{L}}) = z(\mathcal{L}) + z_0 = 2.
\end{array} \right. 
\end{align*}

\underline{If $\langle \tilde{g}^2, L_1^{-1} \tilde{g}^2 \rangle > 0$,} then $\det {\bf D} > 0$ and since 
the second diagonal entry of ${\bf D}$ is again positive for small \(a\), \(\mathbf{D}\) has two positive eigenvalues 
so that $p_0 = 2$ and $z_0 = 0$. Since $n(\mathcal{L}) = 4$ and $z(\mathcal{L}) = 2$ by Proposition \ref{theorem-Morse-second}, 
the count (\ref{index-count}) implies 
\begin{align*}
\left\{ \begin{array}{l} 
n(\hat{\mathcal{L}}) = n(\mathcal{L})-p_0-z_0=4-3=2, \\
z(\hat{\mathcal{L}}) = z(\mathcal{L}) + z_0 = 2.
\end{array} \right. 
\end{align*}
In either case, \( \hat{\mathcal{L}} \) has exactly two negative eigenvalues so that the bifurcating branch is a saddle point of the constrained energy $H$ for fixed momentum $P$ and mass $Q$ with exactly two negative eigendirections. 
In addition, \( \hat{\mathcal{L}} \) has a double zero eigenvalue due to the translational and rotational symmetries of the coupled KdV--LS system (\ref{KdV-NLS}). 
\end{proof}

\begin{remark}
Since the primary branch of the uncoupled KdV solitons is spectrally stable at $\Omega = \widetilde{\Omega}_c$ (and for every $\Omega$), 
we do not know if the bifurcating branch of the coupled solitary waves is spectrally unstable for $\Omega \neq \widetilde{\Omega}_c$. 
Therefore, we cannot use  Theorem 2.4 in \cite{Dmitry} to conclude on the orbital instability of the saddle point of the constrained energy given in Proposition \ref{theorem-stability-second}.
\label{rem-instability}
\end{remark}

\subsection{Spectral instability of the bifurcating branch}

Further to Remark \ref{rem-instability}, we show here that the spectral 
stability problem associated with the bifurcating branch of Proposition \ref{theorem-bifurcation-second} admits a pair of neutrally stable eigenvalues of 
the negative Krein signature embedded into the continuous spectrum 
of the positive Krein signature, see Lemma \ref{lem-instability}. 
The result holds in a range of the values of $k$ which include the interval $\left[ \frac{1}{6},\frac{1}{2}\right]$ in Figure \ref{fig-region-third}, and hence 
the exact solution (\ref{solution-3}) at the bifurcation point. 

\begin{remark}
	We say that the pair of eigenvalues $\lambda = \pm i \omega$ with $\omega \in \mathbb{R}$ of the spectral stability problem with the eigenvectors 
	$\vec{w}_{\pm} = (w,z_1,z_2) \in (H^2(\mathbb{R}))^3$, see (\ref{KdV-NLS-spectral}) below, has negative Krein signature if 
	$\langle \mathcal{L} \vec{w}_{\pm}, \vec{w}_{\pm} \rangle_{L^2} < 0$. It follows from the general theory (see, e.g., \cite{G90} and \cite{CPV05}) that the embedded eigenvalues of the negative Krein signature split into unstable eigenvalues in the continuation of the bifurcating branch past the bifurcation point. However, to be able to claim this splitting, we need to compute a scalar quantity for Fermi's Golden Rule, which is computationally complicated. We postpone further proof of the spectral instability of the bifurcating  branch near $\Omega = \widetilde{\Omega}_c^{(2)}$, and near $\Omega = \Omega_c^{(j)}$ for any $j \geq 2$, for a future work.
\end{remark}

To formulate and prove the claim, we first derive the spectral stability 
problem for the coupled solitary waves with the profile $(U,\Psi)$. 
We use the decomposition 
$$
u(x,t) = U(\xi) + w(\xi,t), \qquad 
\psi(x,t) = e^{-i \omega t} [ \Psi(\xi) + z(\xi,t) ]
$$ 
where $\xi = x - ct$, and obtain the linearization of 
the KdV--LS system (\ref{KdV-NLS}):
\begin{equation}
    \label{KdV-NLS-lin}
\begin{cases}
    w_{t} +  (U w)_{\xi} + w_{\xi \xi \xi} - c w_{\xi} +s (\Psi \bar{z} + \bar{\Psi} z)_{\xi}=0, \\
 i z_{t} + z_{\xi \xi} - i c z_{\xi} + \omega z + k (U z + \Psi w) =0.
\end{cases}
\end{equation}
By using now the variables (\ref{variables-Psi-z}), we get 
\begin{equation}
    \label{KdV-NLS-spectral}
\begin{cases}
    w_{t} = \partial_{\xi} \left[ L_1 w - 2 s A z_1 \right], \\
    z_{1t} = \frac{k}{2 s } L_2 z_2, \\
    z_{2t} = -\frac{k}{2 s } \left[ -2 s A w + L_2 z_1 \right],
\end{cases}
\end{equation}
where $L_1$ and $L_2$ are given in (\ref{block-1}) with $\Omega = \omega + \frac{c^2}{4}$. Separating the variables by using the exponential functions 
$e^{\lambda t}$ in the linearized equations (\ref{KdV-NLS-spectral}), we obtain the spectral stability problem:
\begin{equation}
    \label{KdV-NLS-stability}
\lambda \begin{pmatrix} w \\ z_1 \\ z_2 \end{pmatrix} = \begin{pmatrix} \partial_{\xi} & 0 & 0 \\ 0 & 0 & \frac{k}{2s} \\ 0 & -\frac{k}{2s} & 0 \end{pmatrix} 
\begin{pmatrix} L_1 & -2 s A & 0 \\ -2 s A & L_2 & 0 \\ 0 & 0 & L_2 \end{pmatrix} \begin{pmatrix} w \\ z_1 \\ z_2 \end{pmatrix} \equiv \mathcal{J} \mathcal{L} \lambda \begin{pmatrix} w \\ z_1 \\ z_2 \end{pmatrix},
\end{equation}
where $\mathcal{L}$ is the Hessian operator in Lemma \ref{lem-Hessian} and $\mathcal{J}$ is the skew-adjoint operator satisfying $\mathcal{J}^* = -\mathcal{J}$ in $(L^2(\mathbb{R}))^3$. The continuous spectrum of (\ref{KdV-NLS-stability}) is found from the limit $U(\xi), A(\xi) \to 0$ as $|\xi| \to \infty$ exponentially fast. The continuous spectrum consists of three spectral bands: $i \mathbb{R}$, $i \left[|\Omega|,\infty \right)$ and $i \left( -\infty, -|\Omega| \right]$. There exists no spectral gap near $\lambda = 0$ due to the KdV part of the system (\ref{KdV-NLS-spectral}) and there exists a spectral gap $i \left( -|\Omega|,|\Omega| \right)$ due to the LS part of the system (\ref{KdV-NLS-spectral}). 

\begin{lemma}
	\label{lem-instability}
	Let $(U,A)$ be the profile of the bifurcating branch for $\Omega$ near $\tilde{\Omega}_c$ given by Proposition \ref{theorem-bifurcation-second}. 
	If $k \in (k_-,k_+)$, where 
\begin{equation}
\label{k-minus-plus}
	k_- = \frac{8-5\sqrt{2}}{12} \approx 0.077, \qquad 
	k_+ = \frac{8+5\sqrt{2}}{12} \approx 1.256,
\end{equation}
	then the spectral stability problem (\ref{KdV-NLS-stability}) 
	admits a pair of eigenvalues $\lambda = \pm i \omega$ 
	such that $\omega \in \left( -\infty, -|\Omega| \right] \cup \left[|\Omega|,\infty \right)$ 
	with the corresponding eigenvectors $\vec{w}_{\pm} = (w,z_1,z_2) \in (H^2(\mathbb{R}))^3$ satisfying 
	$\langle \mathcal{L} \vec{w}_{\pm}, \vec{w}_{\pm} \rangle_{L^2} < 0$.
\end{lemma}

\begin{proof}
For the uncoupled KdV solitons (\ref{solution-1}), we have $A = 0$ and the spectral stability problem (\ref{KdV-NLS-stability}) is uncoupled. 
The first component satisfies 
\begin{equation}
    \label{spec-KdV}
\lambda w = \partial_{\xi} L_1 w,
\end{equation}
which is the stability problem for the KdV solitons. Since KdV solitons are spectrally, orbitally, and asymptotically stable \cite{PW94}, 
then $\lambda \in i \mathbb{R}$ for the spectral problem (\ref{spec-KdV}).
The second and third components are diagonalized by the transformation $z_{\pm} = z_1 \pm i z_2$, which satisfy the uncoupled 
spectral problems 
\begin{equation}
    \label{spec-NLS}
\lambda z_{\pm} = \mp \frac{i k}{2 s} L_2 z_{\pm}.
\end{equation}
For any isolated eigenvalue of $L_2$ in the set $\{ \lambda_n \}_{n=1}^J$,
where $J \in \mathbb{N}$ satisfies $k > \frac{J(J-1)}{12}$, 
there exists a pair of purely imaginary eigenvalues $\lambda = \pm i \omega_n$ 
of the spectral problem (\ref{spec-NLS}), 
where $\omega_n = \frac{k}{2} \lambda_n$. The corresponding eigenvector $\vec{w} = (w,z_1,z_2) \in (H^2(\mathbb{R}))^3$ satisfies $\langle \mathcal{L} \vec{w}, \vec{w} \rangle_{L^2} < 0$ if $\lambda_n < 0$ and $\langle \mathcal{L} \vec{w}, \vec{w} \rangle_{L^2} > 0$ if $\lambda_n > 0$.

For the bifurcating branch at $\Omega = \widetilde{\Omega}_c$, 
we have $\lambda_1 < 0$, $\lambda_2 = 0$, and the remaining eigenvalues (if $N \geq 3$) are strictly positive. The eigenvalues $\lambda = \pm i \omega_1$ 
with $\omega_1 = \frac{k}{2} \lambda_1$ belong to the continuous spectrum $i ( -\infty, -|\Omega|] \cup [|\Omega|,\infty )$ if and only if $\lambda_1 < \widetilde{\Omega}_c$. This leads to the constraint $|\Omega_c| > 2 |\widetilde{\Omega}_c|$, rewritten explicitly as 
$$
(\sqrt{1+48k}-1)^2 > 2 (\sqrt{1+48k}-3)^2.
$$
Expanding yields 
$$
1 + 48 k - 10 \sqrt{1 + 48k} + 17 < 0, 
$$
which is satisfied for 
$$
5 - 2\sqrt{2} < \sqrt{1 + 48 k} < 5 + 2\sqrt{2}. 
$$
Further expansion yields $k \in (k_-,k_+)$ with $k_-$, $k_+$ given by  (\ref{k-minus-plus}). 
\end{proof}

\subsection{Numerical approximations of $\langle \tilde{g}^2, L_1^{-1} \tilde{g}^2 \rangle$}

For the numerical approximations of $\langle \tilde{g}^2, L_1^{-1} \tilde{g}^2 \rangle$, we use the scaling transformation
\[
L_{1}^{-1} \tilde{g}^2 =\frac{4}{c}\tilde{W}(y), \quad y = \frac{\sqrt{c}}{2} \xi,
\]
and obtain \(\tilde{W}\) from solutions of the linear inhomogeneous equation 
\begin{equation}
\label{eq-w-2-second}
-\tilde{W}'' + 4 \tilde{W} - 12 {\rm sech}^2(y) \tilde{W} = {\rm sech}^{2q}(y) \tanh^2(y).
\end{equation}
For \(q=1\), the solution of (\ref{eq-w-2-second}) can be computed explicitly as
\[
\tilde{W}(y)=\frac{1}{12}\text{sech}^2(y)(3y\text{tanh}(y)-1),
\]
which yields
\[
\int_\mathbb{R} \tilde{g}^2\tilde{W}dy=\frac{1}{60}. 
\]
For other values of $q$, we obtain $\tilde{W}$ from (\ref{eq-w-2-second}) numerically by using the same representation (\ref{W-expression}) with $g$ replaced by $\tilde{g}$. The numerical approximation of the integral
$$
\int_{\mathbb{R}} \tilde{g}^2 \tilde{W} dy = \frac{\sqrt{c^3}}{8} \langle \tilde{g}^2, L_1^{-1} \tilde{g}^2 \rangle
$$
is plotted in Figure \ref{fig-projection_2} versus $q$. It follows from Figure \ref{fig-projection_2} that the value of $\langle \tilde{g}^2, L_1^{-1} \tilde{g}^2 \rangle$ is positive for \(q<2\) 
and negative for $q>2$. This also agrees with the exact value $\int_\mathbb{R} \tilde{g}^2\tilde{W}dy=\frac{1}{60}$ for $q = 1$ 
and with the type of the subcritical pitchfork bifurcation ($\Omega < \widetilde{\Omega}_c$) suggested by the exact solution (\ref{solution-3}) 
for $k \in \left(\frac{1}{6},\frac{1}{2}\right)$, see Figure \ref{fig-region-third}, where $k \in \left(\frac{1}{6},\frac{1}{2}\right)$ corresponds to $q \in (0,1)$. 
\begin{figure}[htp!]
	\centering
	\includegraphics[width=0.55\textwidth]{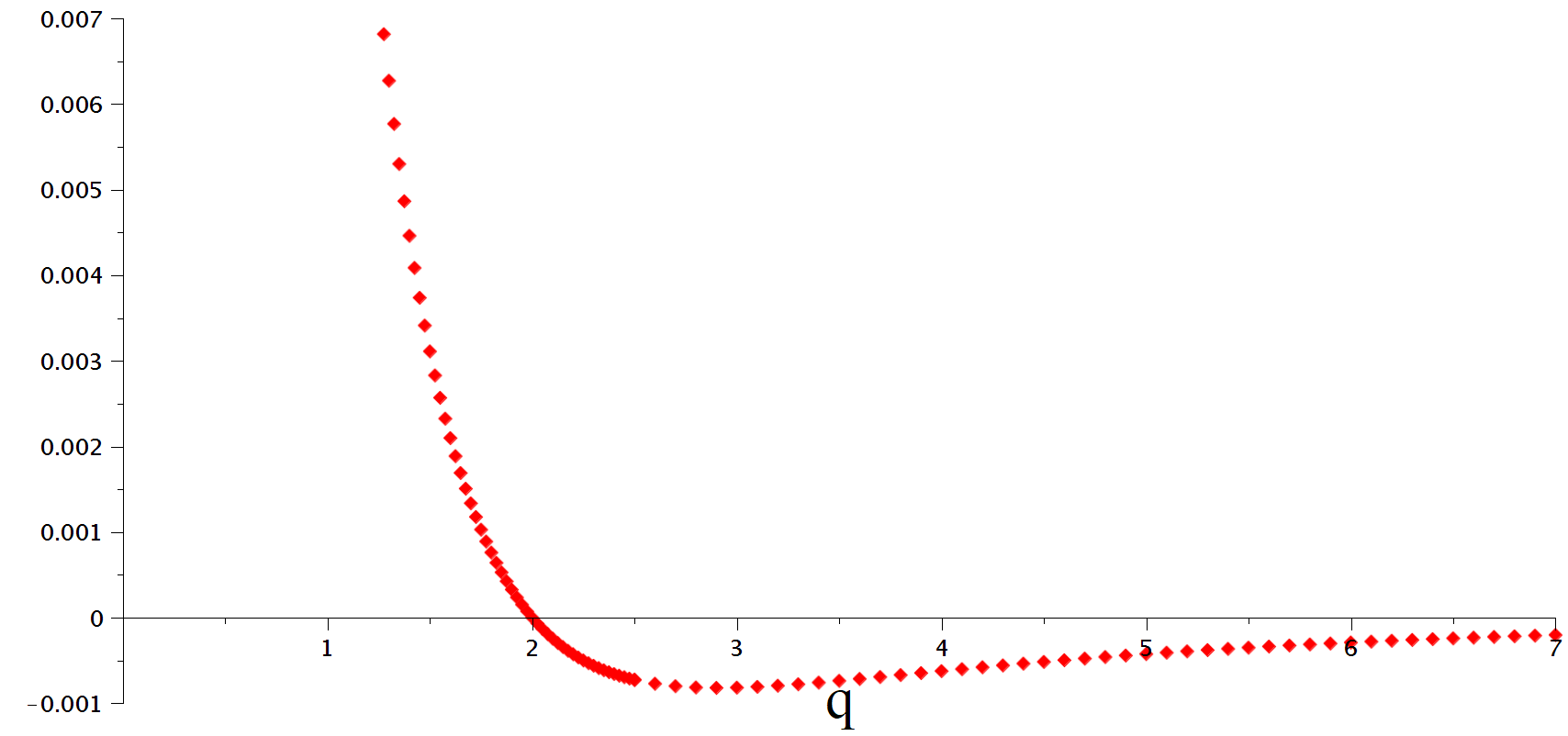}
	\caption{Numerical approximation of $\int_{\mathbb{R}} \tilde{g}^2 \tilde{W} dy$ versus $q$.}
	\label{fig-projection_2}
\end{figure}

{\bf Funding declaration:} The authors declare no funding.

\bibliographystyle{abbrv}
\bibliography{CoupledKdVNLS}

\end{document}